\documentclass[11pt]{article}

\usepackage[papersize={210truemm,297truemm},margin=3.0truecm]{geometry}
\usepackage{wrapfig}
\usepackage{mathrsfs}
\usepackage{amsthm}
\usepackage{amsmath}
\usepackage{amssymb}
\usepackage{amsfonts}
\usepackage[pdftex]{graphicx}  
\usepackage{latexsym}
\usepackage{comment}
\usepackage{ascmac}
\usepackage{caption}
\def\qed{\hfill $\Box$}

\theoremstyle{definition}
\newtheorem{thm}{Theorem}[section]
\newtheorem{lem}[thm]{Lemma}
\newtheorem{cor}[thm]{Corollary}
\newtheorem{df}[thm]{Definition}

\newtheorem{rem}[thm]{Remark}

\newtheorem{ex}[thm]{Example}

\newtheorem*{prf*}{Proof}
\newtheorem{pro}[thm]{Proposition}
\newtheorem*{syuA}{Main Theorem A}
\newtheorem*{syuB}{Main Theorem B}
\newtheorem*{syuC}{Main Theorem C}
\newtheorem*{pf*}{}
\newtheorem*{lem*}{LemmaA}
\newtheorem*{lm*}{LemmaB}
\newtheorem*{not*}{Notation}
\def\mR{\mathbb{R}^3}
\def\g2{l\ge2}
\def\mpt{\mathcal{P}_4}
\def\mpo{\mathcal{P}_6}
\title{Self-similar fractals related to regular tetrahedron and imaginary cubes
 }
\author{Yuto Nakajima\\ Course of Mathematical Science, Department of Human Coexistence, \\
Graduate School of Human and Environmental Studies, Kyoto University\\
Yoshida-nihonmatsu-cho, Sakyo-ku, Kyoto, 606-8501, JAPAN}

\begin{document}
\maketitle
\begin{abstract}
We consider self-similar sets in three-dimensional Euclidean space related to a regular tetrahedron. Sierpi${\rm \acute{n}}$ski tetrahedron is one such self-similar set. In this paper,  we study the whole family of those sets. Our motivation is to obtain three-dimensional analogues of the fractal $n$-gons. In particular, we focus on the geometric properties of those sets from a viewpoint of ``imaginary cube''. An imaginary cube is a set $A$ for which there is some cube $C$ such that the projections of $A$ in the directions of the faces of $C$ equal these projections of $C$. It is already known that the Sierpi${\rm \acute{n}}$ski tetrahedron is an imaginary cube. We obtain a criterion for self-similar sets to be imaginary cubes. Furthermore, we show some properties of those sets which are imaginary cubes from a viewpoint of rotational symmetry or connectedness.
\end{abstract}

\section{Introduction and the main results}
An iterated function system (for short, IFS) is often used to construct a fractal. An IFS $\{\varphi_1,...,\varphi_k\}$ consists of a set of contracting mappings $\varphi_i : X \rightarrow X$, where $X$ is a complete metric space. It is well known (\cite{Fal}, \cite{Hut}) that there uniquely exists a non-empty compact subset $A$ of $X$ such that $A=\bigcup_{i=1}^k\varphi_i(A)$. It is called the attractor or limit set of the IFS. If each $\varphi_i$ is a similarity then the attractor is called a self similar set. Each $\varphi_i(A)$ is often called a {\it piece} of $A$.

We introduce the following self-similar sets.
\begin{df}[Fractal Regular Tetrahedron]
Let $v_1=(1,-1,1),v_2=(-1,1,1), v_3=(-1,-1,-1), v_4=(1,1,-1)$ be the vertices of a regular tetrahedron. Let $c$ be a real number in $(0,1)$ and let $P$ be an element of $SO(3)$. For all $i\in \{1,2,3,4\}$, define $f_i:\mathbb{R}^3\rightarrow{\mathbb{R}^3}$ by $f_i(x)=cPx+v_i$.
Then we call the attractor $A=A(c,P)$ of the IFS $\{f_1, f_2, f_3, f_4\}$ the Fractal Regular Tetrahedron (for short, FRT) generated by $\{f_1, f_2, f_3, f_4\}$.

\end{df}
For examples of FRTs, see Figures \ref{a1}, \ref{a2}, \ref{a3}.

We can find similar settings in \cite{BMS}. In \cite{BMS}, D. Broomhead, J. Montaldi and N. Sidorov consider framework which is similar to the above. However, they deal with self-similar sets generated by IFSs which consist of homothetic functions, that is, the attractors of IFS $\{g_1, g_2, g_3, g_4\}$, where $g_i(x)=cx+v_i$. In this paper, we consider self-similar sets related to regular tetrahedron which are generated by IFSs of affine transformations $\{cPx+v_i\}_{i=1}^4$, where $c\in(0,1), P\in SO(3)$, which have rarely been studied before. In particular, we study the symmetry, the connectedness, and some other properties  of limit sets of IFSs of the above type which are related to ``imaginary cubes''(an imaginary cube is a set $A$ for which there is some cube $C$ such that the projections of $A$ in the directions of the faces of $C$ equal these projections of $C$. See Definition \ref{ic}).

An FRT is the three-dimensional analogue of fractal $n$-gon (\cite{BanHu}, \cite{BH}, See Figures \ref{a4}, \ref{a5}, \ref{a6}). A fractal $n$-gon $A(\lambda)$ is the attractor of an IFS $\{\phi_1,...,\phi_n\}$ where, $\phi_i: \mathbb{C} \rightarrow \mathbb{C}$ is defined by $\phi_i(z)=\lambda z+b_i$, $\lambda \in \mathbb{D}^{\times}\equiv \{ z\in \mathbb{C}| 0<|z|<1\}$ and $b_1,...,b_n$ are vertices of a regular $n$-gon. Since a fractal $n$-gon possesses a rotational symmetry of order $n$, if there exists $i\in\{1, 2, ..., n-1\}$ such that $\phi_i(A(\lambda))\cap \phi_{i+1}(A(\lambda))\neq \emptyset$, then for each $i\in \{1, 2, ..., n-1\}$, we have $\phi_i(A(\lambda))\cap \phi_{i+1}(A(\lambda))\neq \emptyset$ and the fractal $n$-gon is connected (\cite{BanHu}, \cite{Hata}). Furthermore, if there exists $i\in\{1, 2, ..., n-1\}$ such that $\phi_i(A(\lambda))\cap \phi_{i+1}(A(\lambda))=\emptyset$, the fractal $n$-gon is not connected (\cite{BanHu}). Hence a fractal $n$-gon $A(\lambda)$ is connected if and only if there exists $i\in\{1, 2, ..., n-1\}$ such that $\phi_i(A(\lambda))\cap \phi_{i+1}(A(\lambda))\neq \emptyset$ (\cite{BanHu}, Theorem 2). 

But FRTs do not satisfy such properties in general (Example \ref{ex}). This example raises a question about rotational symmetries of FRTs. Here, for a set $A\subset \mR$, the rotational symmetry of $A$ is defined as the set $\{Q\in SO(3)| QA=A\}$ and it is denoted by $symA$. We define some terminologies to study the rotational symmetries of FRTs. In this paper, we denote by $C_0$ the cube whose vertices contain $v_1,v_2,v_3,v_4$ and $T_0$ the regular tetrahedron whose vertices are $v_1,v_2,v_3,v_4$. Recall that $v_1,v_2,v_3,v_4$ are used in the definition of FRTs. We say that $C\subset \mR$ is a cube if there exist $a\in \mathbb{R}$, $v\in \mR$ and $Q\in SO(3)$ such that $C=aQC_0+v$.
Put $\mathcal{P}_4:=symT_0,\mathcal{P}_6:=symC_0$. We call $\mathcal{P}_4$ the tetrahedral group and $\mathcal{P}_6$ the hexahedral group.

The following is one of the main results in this paper.
\begin{syuA}(Theorem 4.4 (2))
Let $c\in (0,1)$ and $P\in \mpo$. Then $symA(c,P)= \mathcal{P}_4$.
\end{syuA}
It is natural to study when the rotational symmetry of an FRT is equal to $\mathcal{P}_4$. The proof of the following result is based on methods used in \cite{Fal2}. We can find the results about the symmetry of self-similar sets in \cite{Fal2}. 
\begin{syuB}(Theorem 4.5)
Let $c\in(0,\frac{\sqrt{2}}{\sqrt{2}+\sqrt{3}}]$ and $P\in{SO(3)}\backslash{\mathcal{P}_6}$. Then $symA(c,P)\neq{\mathcal{P}_4}$.
\end{syuB}
Main Theorems A and B are related to the study of affine embeddings of self-similar sets. For the study of affine embeddings of self-similar sets, see \cite{A}, \cite{EKM}, \cite{FW}, \cite{FHR}.

But it is difficult to study the rotational symmetries of FRTs when $c$ is large in general. When we observe FRTs $A(c,P)$, where $c\in (0,1)$ and $P\in \mpo$, we can see that connected FRTs are characterized as $Imaginary$ $cube$s as in the following theorem which is the third main result. 
\begin{syuC}(Theorem \ref{decide})
Let $A(c,P)$ be a FRT. Then $A(c,P)$ is an imaginary cube if and only if $(c,P)\in[1/2,1)\times{\mathcal{P}_6}$.
\end{syuC}
\begin{figure}[h]
  \begin{center}
    \begin{tabular}{c}
\hspace{-1.8cm}
      \begin{minipage}{0.4\hsize}
        \begin{center}
          \includegraphics[clip, width=5cm]{H01}
         \caption{A(1/2, E). This is a well-known fractal called the Sierpi${\rm \acute{n}}$ski tetrahedron.}
     \label{a1}

        \end{center}
      \end{minipage}
\hspace{0.3cm}
   
      \begin{minipage}{0.4\hsize}
        \begin{center}
          \includegraphics[clip, width=5cm]{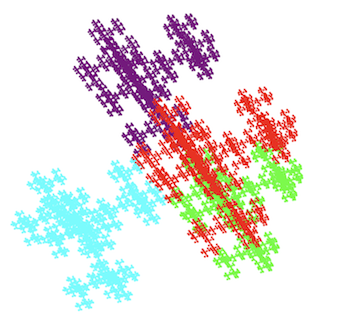}
         \caption{$A(c,P)$, where $c=1/2$, $P$ is the $\pi/2$ rotation matrix around $l_1$. Here, $l_1$ is the unique line containing the origin and $(0,0,1)$.}
       \label{a2}
    
        \end{center}
      \end{minipage}
      \hspace{0.2cm}
       \begin{minipage}{0.4\hsize}
        \begin{center}
          \includegraphics[clip, width=5cm]{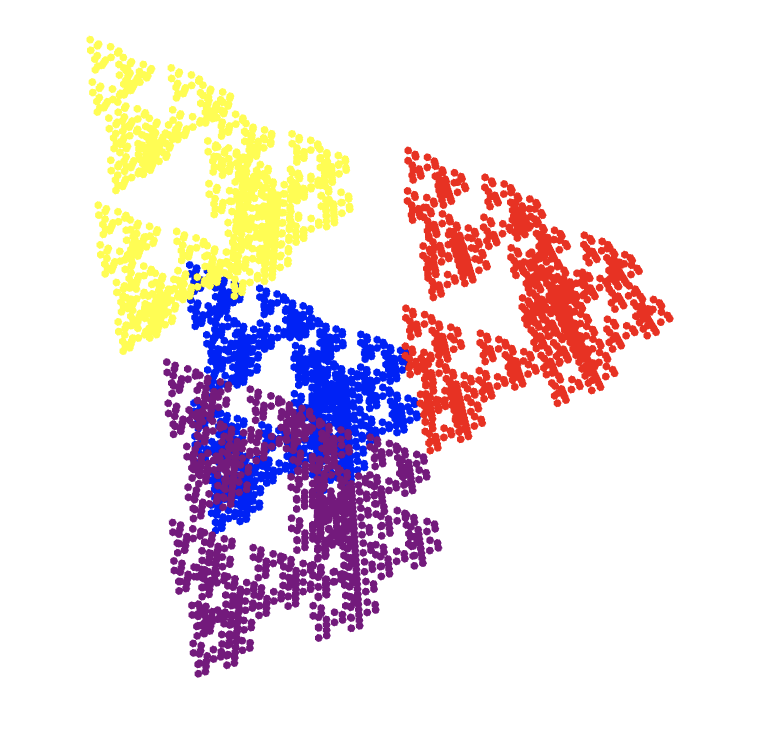}
         \caption{$A(c,P)$, where $c=1/2$, $P$ is the $\pi/3$ rotation matrix around $l_2$. Here, $l_2$ is the unique line containing the origin and $v_1$.}
     \label{a3}

        \end{center}
      \end{minipage}

    \end{tabular}

  \end{center}
\end{figure}

\begin{figure}[h]
  \begin{center}
    \begin{tabular}{c}
    \hspace{-1.8cm}
     \begin{minipage}{0.4\hsize}
        \begin{center}
          \includegraphics[clip, width=5cm]{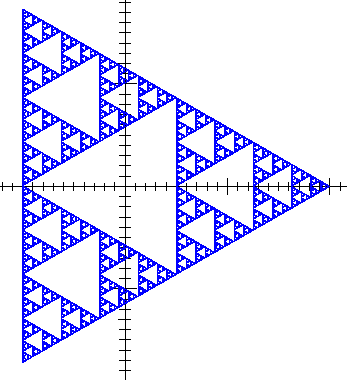}
         \caption{fractal 3-gon}
     \label{a4}

        \end{center}
      \end{minipage}

      \begin{minipage}{0.4\hsize}
        \begin{center}
          \includegraphics[clip, width=5cm]{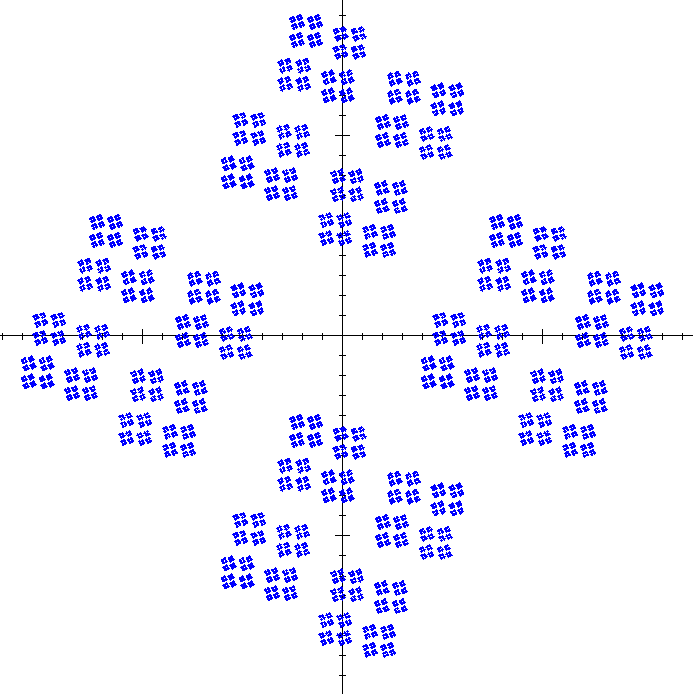}
         \caption{fractal 4-gon}
     \label{a5}

        \end{center}
      \end{minipage}

      \begin{minipage}{0.4\hsize}
        \begin{center}
          \includegraphics[clip, width=5cm]{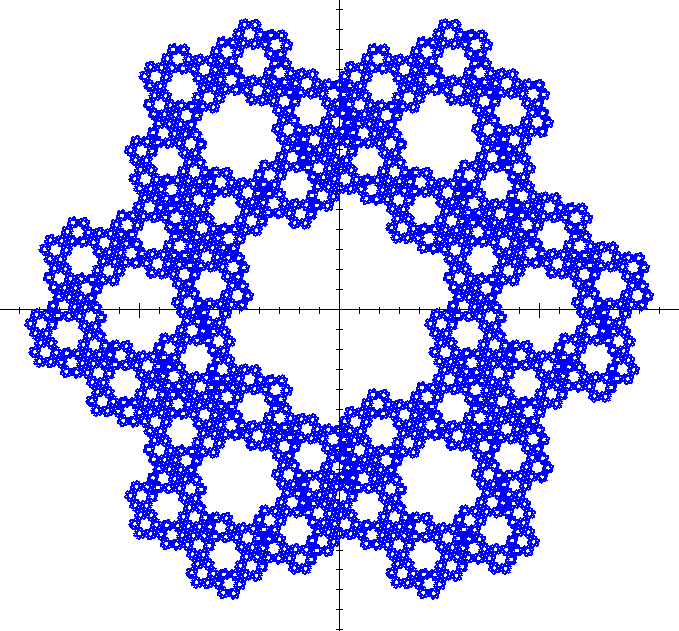}
         \caption{fractal 6-gon}
       \label{a6}
    
        \end{center}
      \end{minipage}

    \end{tabular}

  \end{center}
\end{figure}

Imaginary cube of a cube $C$ (for short IC of $C$) is the three-dimensional object which has the same square projections in three-orthogonal directions just as a cube $C$ has. (See Definition \ref{ic} and see Figure \ref{aa13} for an example.) We can see many results about ICs in \cite{T}, \cite{T2}, \cite{T3}, \cite{T4}, \cite{T5}, \cite{TT}, \cite{TT2} and \cite{TY}. An IC of a cube $C$ is often simply called an IC. In 2007, H. Tsuiki initiated the study of imaginary cubes (\cite{T}). Sierpi${\rm \acute{n}}$ski tetrahedron, which is an FRT $A(1/2, E)$ where $E$ is the identity matrix, is an imaginary cube (\cite{T}). It is interesting that Sierpi${\rm \acute{n}}$ski tetrahedron, which is a three-dimensional object, has square projections although it is ``thin'' in the sense that the Hausdorff dimension of Sierpi${\rm \acute{n}}$ski tetrahedron is equal to 2 which is less than $3={\rm dim}{\mathbb{R}^3}$. He considered other self-similar sets in three-dimensional space with the same properties as the Sierpi${\rm \acute{n}}$ski tetrahedron. For each $k\ge 2$, he investigated self-similar sets $A$ of IFSs for which the followings hold.
\begin{enumerate}
\item {\textit{Imaginary cube}}: The self-similar set $A$ is a set for which there is some cube $C$ such that the projections of $A$ in the directions of the faces of $C$ equal these projections of $C$.
\item $A$ is the union of $k^2$ copies of itself with $1/k$ scale.
\item The similarity transformations in the IFS do not include rotational parts.
\end{enumerate}
H. Tsuiki calculated the number of non-congruent self-similar sets which satisfy the condition above for each $k\le5$ (\cite{T}). 
In this paper, we consider self-similar sets generated by IFSs of similarity transformations which include rotational parts and we obtain a criterion for the self-similar sets to be imaginary cubes (Theorem \ref{decide}). 
\\  
\begin{figure}[h]
\begin{center}
\includegraphics[scale=0.5]{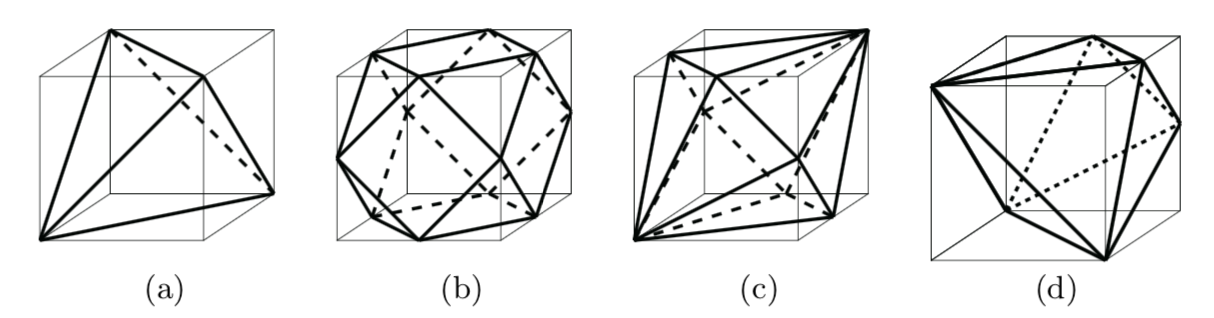}
\caption{Examples of ICs(\cite{T})}
\label{aa13}
\end{center}
\end{figure}

\subsection*{Acknowledgement.} 
The author would like to express his gratitude Professor Hiroki Sumi and Professor Hideki Tsuiki for their valuable comments. The author thank Rich Stankewitz for helpful comments. This study is supported by JSPS KAKENHI Grant Number JP 19J21038.

\section{Preliminaries}
Let $\{\varphi_1,...,\varphi_k\}$ be a general IFS on a complete metric space $X$. We define the address map as follows. Let $I=\{1,2,..., k\}$. For each $\omega=\omega_{1}\omega_{2}\omega_{3} \cdots \in I^{\infty}$, we set $\omega|_{n} := \omega_{1}\omega_{2} \cdots \omega_{n} \in I^{n}$ and $\varphi_{\omega|_{n}} := \varphi_{\omega_{1}} \circ \varphi_{\omega_{2}} \circ \cdots \circ \varphi_{\omega_{n}}$. Then, it is well known that for each $\omega\in I^{\infty}$, $\lim_{n\to \infty}\varphi_{\omega|_{n}}(x)\in X$ exists, where $x$ is an arbitrary point in $X$. Note that this limit does not depend on the choice of $x$. It is denoted by $v_{\omega}$. The address map $p \colon I^{\infty} \rightarrow X$ is defined by $\omega \mapsto v_{\omega}$. Note that $p(I^{\infty})=A$, where $A$ is the limit set generated by $\{\varphi_1,..., \varphi_k\}$. If $p(\omega)=v$, then $\omega$ is called an address of $v$.  In the following, for each finite word $\omega=\omega_1\cdots\omega_n$, we set $\overline{\omega}=\omega_1\cdots\omega_n\omega_1\cdots\omega_n\omega_1\cdots\omega_n\cdots\in I^{\infty}$.

Let $c\in(0,1)$, $P\in SO(3)$. For each $i\in \{1,2,...\}$,  we set $P^i=\overbrace{P\cdot P\cdots P}^{i \rm{times}}$ and $P^0=E$, where $E$ is the three-dimensional identity matrix. Note that $(cP)^i=c^iP^i$.

In our case, the address map has a particularly simple form as we see in the following. 
\begin{rem}
\label{ad}
The point $v$ with address $\omega_1\omega_2\cdots\omega_n\cdots\in \{1,2,3,4\}^{\infty}$ in the FRT $A(c,P)$ has the representation $v = v_{\omega_1}+\sum_{i=1}^{\infty}(cP)^{i}v_{\omega_{i+1}}$. 
\end{rem}

\begin{proof}
For each $n\ge2$, we have that $f_{\omega_{1}} \circ f_{\omega_{2}} \circ \cdots \circ f_{\omega_{n}}=v_{\omega_1}+\sum_{i=1}^{n-1}(cP)^{i}v_{\omega_{i+1}}$. 
\end{proof}
 Let $A$ be an FRT.  We set $A_i=f_i(A)$ for each $i\in \{1,2,3,4\}$. It is difficult to study whether the FRT has the rotational symmetry which acts transitively on the pieces, permuting them as vertices of regular tetrahedron, but their pieces move to other pieces in parallel as we see in the following.
\begin{lem}
\label{trns}
For all $i, j \in \{1,2,3,4\}$ with $i\neq j$, we have $A_i+v_j-v_i=A_j$.
\end{lem}

\begin{proof} Fix $c\in(0,1)$ and $P\in SO(3)$. Then

\begin{align*}
A_i+v_j-v_i
&=\{v_i+\sum_{l=1}^{\infty}(cP_1)^{l}v_{\omega_{l}}|\omega=\omega_1\omega_2\cdots\in{\{1,2,3,4\}^{\infty}}\}+v_j-v_i\\
&=\{v_j+\sum_{i=1}^{\infty}c^{l}v_{\omega_{l}}|\omega=\omega_1\omega_2\cdots\in{\{1,2,3,4\}^{\infty}}\}\\
&=A_j.
\end{align*}
We have thus proved our proposition. 
\end{proof}
The following example explains how difficult the pieces of an FRT intersect other pieces.
\begin{ex}
\label{ex}
Let $l$ be the unique line containing the origin and $v_1$. Let $P$ be the $\pi/3$ rotation matrix around $l$ and $A=A(1/2,P)$ an FRT(See Figures \ref{i.png}, \ref{j.png}). Then we have $A_1\cap{A_i}=\emptyset$ and $A_i\cap{A_j}\neq{\emptyset}$ for all $i,j\in\{2,3,4\}$.
\end{ex}
\begin{figure}[h]
  \begin{center}
    \begin{tabular}{c}

      \begin{minipage}{0.4\hsize}
        \begin{center}
          \includegraphics[clip, width=5cm]{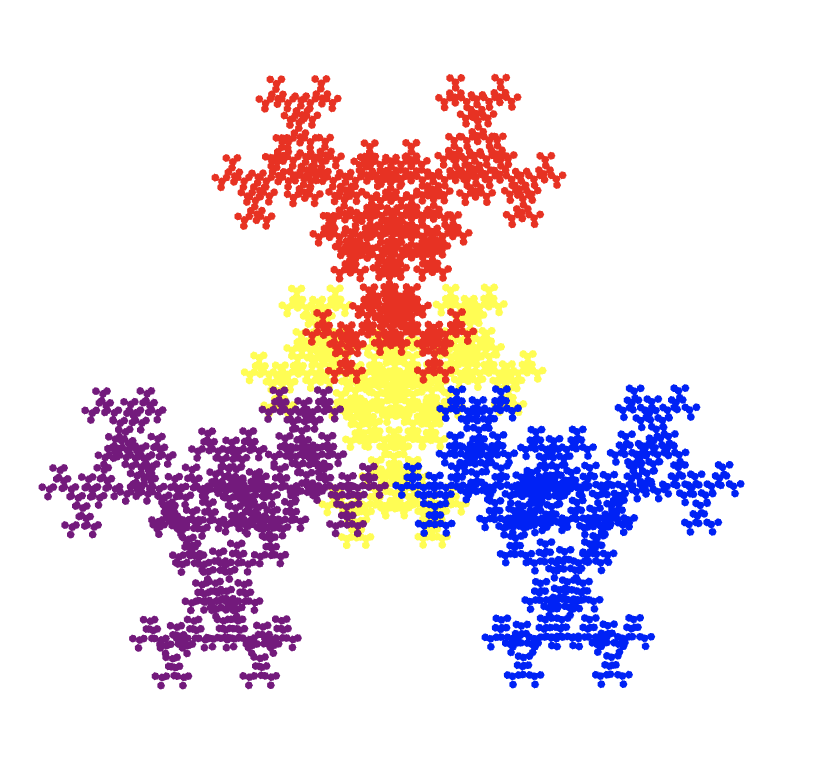}
         \caption{The red piece intersects the blue one , the blue one intersects the purple one and the purple one intersects the red one.}
     \label{i.png}

        \end{center}
      \end{minipage}

   \hspace{0.2cm}
      \begin{minipage}{0.4\hsize}
        \begin{center}
          \includegraphics[clip, width=5cm]{j.png}
         \caption{The yellow piece does not intersect red one, purple one or blue one.}
       \label{j.png}
    
        \end{center}
      \end{minipage}

    \end{tabular}

  \end{center}
\end{figure}

\begin{proof}
Let $H$ be the coordinate change defined by $$H(x, y, z)=(\frac{-\sqrt{2}}{2\sqrt{3}}x+\frac{\sqrt{2}}{2\sqrt{3}}y+\frac{\sqrt{2}}{\sqrt{3}}z, \frac{-\sqrt{2}}{2}x+\frac{-\sqrt{2}}{2}y, \frac{2}{2\sqrt{3}}x+\frac{-2}{2\sqrt{3}}y+\frac{1}{\sqrt{3}}z).$$
We change coordinates of points by using the coordinate change $H$ so that 
\begin{align*}
H(v_1)=(0,0,\sqrt{3}), H(v_2)=(\frac{2\sqrt{2}}{\sqrt{3}},0,\frac{-1}{\sqrt{3}}),H(v_3)=(\frac{-\sqrt{2}}{\sqrt{3}},\sqrt{2},\frac{-1}{\sqrt{3}}), H(v_4)=(\frac{-\sqrt{2}}{\sqrt{3}},-\sqrt{2},\frac{-1}{\sqrt{3}}).
\end{align*}
In this proof, we set $h_i:=H(v_i)$, $Q:= HPH^{-1}$ and $A^{\prime}_i:=H(A_i)$ for each $i\in \{1,2,3,4\}$. Note that 
\begin{align*}
Qh_1=h_1, Q^2h_2=h_3, Q^2h_3=h_4, Q^2h_4=h_2.
\end{align*}
Put $\Delta=\bigr\{h_i-h_j|i,j\in\{1,2,3,4\}\bigl\}$. First we prove $A^{\prime}_1\cap A^{\prime}_2={\emptyset}$. Suppose that there is a point $v\in A^{\prime}_1\cap A^{\prime}_2$. By Remark \ref{ad}, the point $v\in A^{\prime}_1\cap A^{\prime}_2$ can be represented as 
\begin{align*}
v = h_{1}+\sum_{i=1}^{\infty}(\frac{1}{2})^iQ^{i}h_{\omega_{i}} = h_{2}+\sum_{i=1}^{\infty}(\frac{1}{2})^iQ^{i}h_{\tau_{i}}
\end{align*}

where $\omega_i, \tau_i \in \{1,2,3,4\}$. Thus 
\begin{align*}
h_1-h_2+\sum_{i=1}^{\infty}(\frac{1}{2})^iQ^id_i=0,
\end{align*}

where $d_i \in \Delta$. We denote by $x(v), y(v)$ and $z(v)$  the first, second and third coordinate of $v$ respectively. For example, $x(h_1)=0, y(h_3)=\sqrt{2}$. Since $z(h_1-h_2)={4}/{\sqrt{3}}$, 
\begin{align*}
z(\sum_{i=1}^{\infty}(\frac{1}{2})^iQ^id_i)=\sum_{i=1}^{\infty}(\frac{1}{2})^iz(d_i)=\frac{-4}{\sqrt{3}}.
\end{align*} 

Note that $z(d)\in\{{-4}/{\sqrt{3}}, 0, {4}/{\sqrt{3}}\}$ for all $d\in \Delta$. Hence we obtain that $d_i\in\{h_2-h_1,h_3-h_1,h_4-h_1\}$ for all $i$. Since $x(h_1-h_2)={-2\sqrt{2}}/{\sqrt{3}}$,
\begin{align*}
x(\sum_{i=1}^{\infty}(\frac{1}{2})^iQ^id_i)=\sum_{i=1}^{\infty}(\frac{1}{2})^ix(Q^id_i)=\frac{2\sqrt{2}}{\sqrt{3}}.
\end{align*}

Note that $x(Q^id_i)\le {2\sqrt{2}}/{\sqrt{3}}$ for all $i$. Hence we obtain that $x(Q^id_i)={2\sqrt{2}}/{\sqrt{3}}$ for all $i$. But this contradicts $x(Q(h_2-h_1)), x(Q(h_3-h_1)), x(Q(h_4-h_1))<{2\sqrt{2}}/{\sqrt{3}}$. Thus $A^{\prime}_1\cap A^{\prime}_2={\emptyset}$. Similarly $A^{\prime}_1\cap A^{\prime}_3={\emptyset}, A^{\prime}_1\cap A^{\prime}_4={\emptyset}$.
Next we prove $A^{\prime}_2\cap A^{\prime}_3\neq{\emptyset}$. Note that $\{h_i-h_j|i,j\in\{2,3,4\}\}$ is the set of all vertices of a regular hexagon.

\begin{align*}
p(2\overline{322443})-p(3\overline{443322})&=h_2-h_3+\frac{1}{2}Q\sum_{i=0}^{\infty}(\frac{1}{2})^{6i}(h_3-h_4)+(\frac{1}{2}Q)^2\sum_{i=0}^{\infty}(\frac{1}{2})^{6i}(h_2-h_4)\\&\qquad\qquad\hspace{0.15cm}+(\frac{1}{2}Q)^3\sum_{i=0}^{\infty}(\frac{1}{2})^{6i}(h_2-h_3)+(\frac{1}{2}Q)^4\sum_{i=0}^{\infty}(\frac{1}{2})^{6i}(h_4-h_3)\\&\qquad\qquad\hspace{0.15cm}+(\frac{1}{2}Q)^5\sum_{i=0}^{\infty}(\frac{1}{2})^{6i}(h_4-h_2)+(\frac{1}{2}Q)^6\sum_{i=0}^{\infty}(\frac{1}{2})^{6i}(h_3-h_2)\\
&=h_2-h_3+\frac{1}{2}\sum_{i=0}^{\infty}(\frac{1}{2})^{6i}(h_3-h_2)+(\frac{1}{2})^2\sum_{i=0}^{\infty}(\frac{1}{2})^{6i}(h_3-h_2)
\\&\qquad\qquad\hspace{0.15cm}+(\frac{1}{2})^3\sum_{i=0}^{\infty}(\frac{1}{2})^{6i}(h_3-h_2)+(\frac{1}{2})^4\sum_{i=0}^{\infty}(\frac{1}{2})^{6i}(h_3-h_2)\\&\qquad\qquad\hspace{0.15cm}+(\frac{1}{2})^5\sum_{i=0}^{\infty}(\frac{1}{2})^{6i}(h_3-h_2)+(\frac{1}{2})^6\sum_{i=0}^{\infty}(\frac{1}{2})^{6i}(h_3-h_2)=0.
\end{align*}

Thus $A^{\prime}_2\cap A^{\prime}_3\neq{\emptyset}$. Similarly we have $A^{\prime}_3\cap A^{\prime}_4\neq{\emptyset}$ and $A^{\prime}_2\cap A^{\prime}_4\neq{\emptyset}$.

\end{proof}

\section{A criterion for FRTs to be imaginary cubes}
We give a rigorous definition of  $Imaginary$ $Cube$ as follows.
\begin{df}[Imaginary Cube]
\label{ic}
Let $A$ be a subset of ${\mathbb{R}^3}$ and $C$ a cube. Let $S_1,S_2,S_3$ be faces which share a vertex of $C$. For all $i=1,2,3$, let $\pi_i:\mathbb{R}^3\rightarrow{L_i}$ be the orthogonal projection of $\mathbb{R}^3$ onto $L_i$, where $L_i$ is the two-dimensional vector space of $\mathbb{R}^3$ parallel to $S_i$. Then
\begin{enumerate}
\item We say that $A$ is an $Imaginary$ $Cube$ (for short, IC) of $C$ if $\pi_i(A)=\pi_i(C)$ for all ${i}=1,2,3$.
\item We say that $A$ is an IC if there exists a cube $C$ such that $A$ is an IC of $C$.
\end{enumerate}
\end{df}

\begin{rem}
\label{rem}
Let $A$ be a subset of ${\mathbb{R}^3}$ and let $C$ be a cube. If $A$ is an $IC$ of $C$, it follows that
　\begin{itemize}
\item$A\subset{C}.$
\item there exists a point of ${A}$ on each edge of $C$ .
\end{itemize}
\end{rem}
The following lemma gives a sufficient condition for imaginary cubes.
\begin{lem}
\label{ic2}
Let $\{B_i\}_{i=1}^{\infty}$ be a decreasing sequence of non-empty compact subsets of $\mathbb{R}^3$ and $A=\bigcap_{i=1}^{\infty}B_i\neq{\emptyset}$. Suppose that there exist a cube $C$ such that $B_i$ is an IC of $C$ for all $i$. Then $A$ is an IC of $C$.
\end{lem}
\begin{prf*}
Let $\pi : A \rightarrow S$ be one of the three square projections of $A$ to a square $S$. Then we have that $\pi(A)\subset S$ is trivial. Fix $x\in S$ and let $D_i = \pi^{-1}(x) \cap B_i$ for each $i = 0,1,2,...$. Since $\{B_i\}_{i=1}^{\infty}$ is a decreasing sequence of non-empty compact subsets of $\mathbb{R}^3$, We have that $\{D_i\}_{i=1}^{\infty}$ is a decreasing sequence of non-empty compact subsets of $\mathbb{R}^3$. Hence we have that $\pi^{-1}(x) \cap A=\bigcap_{i=1}^{\infty}D_i$ is not empty. Then we have that $S\subset \pi(A)$. Hence we have proved our lemma. \qed
\end{prf*}
The following is one of the main results in this paper.
\begin{thm}
\label{decide}
Let $A(c,P)$ be a FRT. Then $A(c,P)$ is an IC if and only if $(c,P)\in[1/2,1)\times{\mathcal{P}_6}$.
\end{thm}

The proof is divided into the proofs of the following three propositions.

\begin{pro}
\label{decide1}
Let $c\in[1/2,1),P\in{\mathcal{P}_6}$ and $A(c,P)$ a FRT. Then $A(c,P)$ is an IC.
\end{pro}

\begin{pro}
\label{decide2}
Let $c\in(0,1/2),P\in{SO(3)}$ and $A(c,P)$ a FRT. Then $A(c,P)$ is not an IC.
\end{pro}

\begin{pro}
\label{decide3}
Let $c\in(0,1),P\in{SO(3)}\backslash{\mathcal{P}_6}$ and $A(c,P)$ a FRT. Then $A(c,P)$ is not an IC.
\end{pro}

\begin{pf*}[proof of Proposition \ref{decide1}]
Fix $c\in [1/2,1),P\in{\mathcal{P}_6}$.

For all nonempty compact subset $K$, put $\Phi(K)=\bigcup_{j=1}^4f_j(K)$. Let $C$ be the cube whose vertices contain $\{{v_j}/({1-c})\}_{j=1}^4$ and $B_i={\Phi}^i(C)$. We show that for each $i$, $B_i$ is an IC of $C$. This is done by induction. Suppose that $B_i$ is an IC of $C$. Then for each $j=1,2,3,4$, $f_j(B_i)$ is an IC of $f_j(C)$. Since $P$ preserves $C$, $f_j(C)$ is the cube whose vertices contain $\{\frac{cv_l}{1-c}+v_j\}_{l=1}^4$. It follows that $f_1(C), f_2(C), f_3(C), f_4(C)$ are cubes whose vertices contain 
\begin{align*}
&\{(\frac{1}{1-c},\frac{-1}{1-c},\frac{1}{1-c}),(\frac{1-2c}{1-c},\frac{2c-1}{1-c},\frac{1}{1-c}),(\frac{1-2c}{1-c},\frac{-1}{1-c},\frac{1-2c}{1-c}),(\frac{1}{1-c},\frac{2c-1}{1-c},\frac{1-2c}{1-c})\},\\
&\{(\frac{-1}{1-c},\frac{1}{1-c},\frac{1}{1-c}),(\frac{2c-1}{1-c},\frac{1-2c}{1-c},\frac{1}{1-c}),(\frac{2c-1}{1-c},\frac{1}{1-c},\frac{1-2c}{1-c}),(\frac{-1}{1-c},\frac{1-2c}{1-c},\frac{1-2c}{1-c})\},\\
&\{(\frac{-1}{1-c},\frac{-1}{1-c},\frac{-1}{1-c}),(\frac{2c-1}{1-c},\frac{2c-1}{1-c},\frac{-1}{1-c}),(\frac{2c-1}{1-c},\frac{-1}{1-c},\frac{2c-1}{1-c}),(\frac{-1}{1-c},\frac{2c-1}{1-c},\frac{2c-1}{1-c})\}\\&\hspace{7.3cm} \mbox{and}\\
&\{(\frac{1}{1-c},\frac{1}{1-c},\frac{-1}{1-c}),(\frac{1-2c}{1-c},\frac{1-2c}{1-c},\frac{-1}{1-c}),(\frac{1}{1-c},\frac{1-2c}{1-c},\frac{2c-1}{1-c}),(\frac{1-2c}{1-c},\frac{1}{1-c},\frac{2c-1}{1-c})\}\\
\end{align*}
respectively. Combining this with the assumption $c\ge 1/2$, we obtain that for each $i, j\in \{1,2,3,4\}$, $f_i(C)\cap f_j(C)\neq \emptyset$ and $\bigcup_{j=1}^4 f_j(C)$ is an IC of $C$ (for the case $c=1/2$, see the figure of $B_1$ in Figure \ref{i-c4}). Thus $B_{i+1}$ is an IC of $C$. 
Then $C$, $\{B_i\}_{i=1}^{\infty}$ and FRT $A(c,P)$ satisfy the assumption of Lemma \ref{ic2} (see Figure \ref{i-c4}). \qed
\end{pf*}
\begin{figure}[h]
\begin{center}
\includegraphics[scale=0.3]{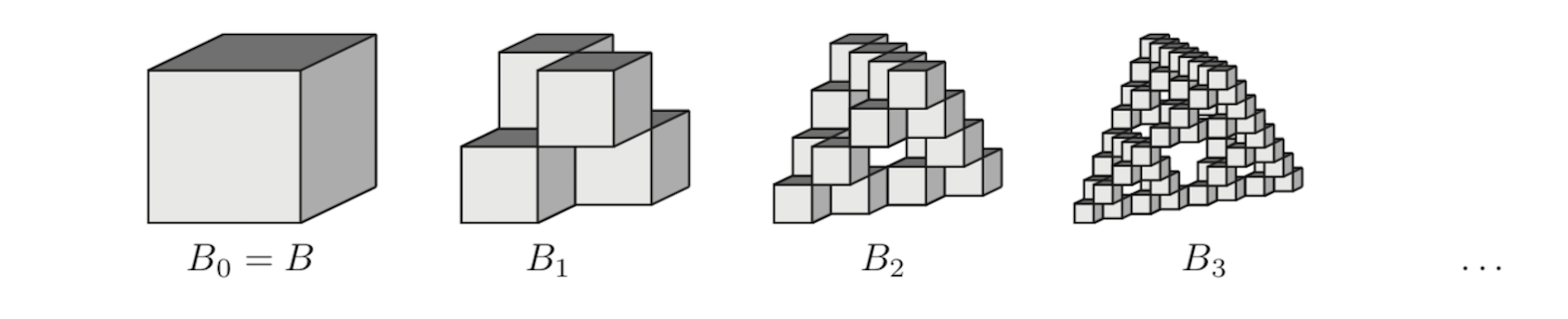}
\caption{The case $c=1/2$ and $P$ is the identity matrix.(\cite{T})}
\label{i-c4}
\end{center}
\end{figure}
\begin{pf*}[proof of Proposition \ref{decide2}]
Fix $c\in (0,1/2),P\in{SO(3)}$. We have the following.

\begin{itemize}
\item Let $A\subset{\mathbb{R}^3}$, let $L$ be the two-dimensional vector space of $\mathbb{R}^3$ and let $\pi_{L}:\mathbb{R}^3\rightarrow{L}$ be the orthogonal projection of $\mathbb{R}^3$ onto $L$. We denote by ${\rm dim}_H(A)$ the Hausdorff dimension of $A$ with respect to the Euclidean distance. Then 
\begin{align*}
{\rm dim}_H(\pi_{L}(A))\le{{\rm min}\{{\rm dim}_H(A),2\}}.
\end{align*}
by the Lipschitz continuity of $\pi_{L}$$ (see \cite[p. 98]{Fal})$.
\item The Hausdorff dimension is less than the similarity dimension, that is ${\rm dim}_H(A(c,P))\le{-\log4}/{\log{c}}<2$$ (see \cite[p. 140]{Fal})$.
\end{itemize}
As can be seen from the claims above, it follows that $\pi_{L}(A(c,P))$ is not a rectangle for all two-dimensional vector spaces $L$.\qed

\end{pf*}
In order to prove Proposition \ref{decide3}, we need the following terminologies and lemmas 3.9, 3.10, 3.11.
\begin{df}[positive side and negative side]
Let $S$ be a face, let $u$ be a vector which is normal to $S$, and let $a\in{\mathbb{R}^3}$ as in Figure \ref{aaa1}. Then we say that $a$ belongs to the positive (resp. negative) side of $S$ (with respect to $u$) if there exist $x_0\in{S}$ such that $(u,a-x_0)>0$(resp. $<0)$. Note that this definition is independent of choice of $x_0$.
\end{df}
\begin{figure}[h]
  \begin{center}
    \begin{tabular}{c}

      \begin{minipage}{0.4\hsize}
        \begin{center}
          \includegraphics[clip, width=5cm]{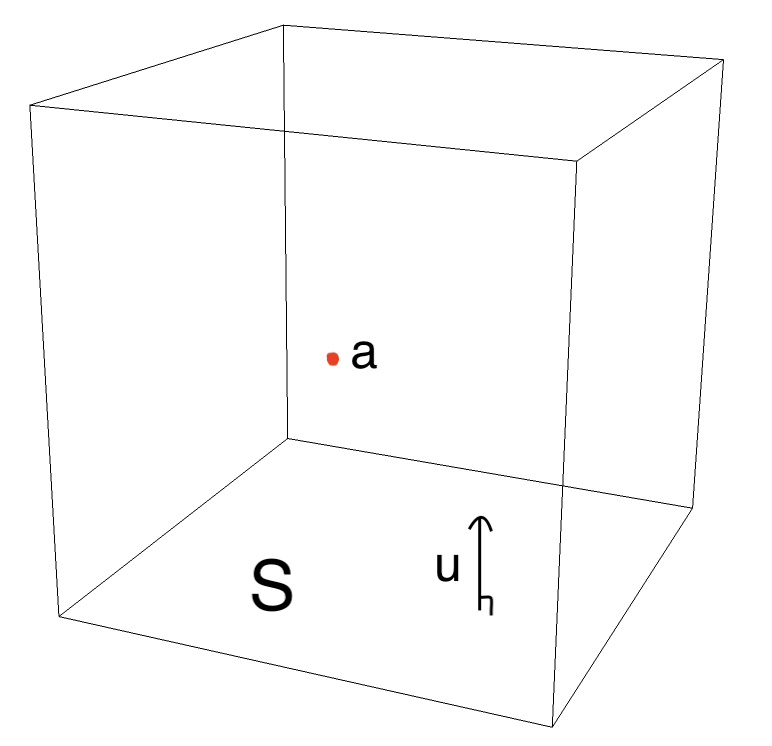}
         \caption{The point $a$ belongs to the positive side of $S$ with respect to $u$.}
     \label{aaa1}

        \end{center}
      \end{minipage}

    \end{tabular}

  \end{center}
\end{figure}
In this paper,  we say that $C$ is a normal cube if there exist $b\in \mathbb{R}$ and $w\in \mathbb{R}^3$ such that $C=bC_0+w$. Recall that $C_0$ is the cube whose vertices contain $v_1,v_2,v_3,v_4$ and $T_0$ is the regular tetrahedron whose vertices are $v_1,v_2,v_3,v_4$.
\begin{lem}
\label{ho1}
Let $A=A(c,P)$ be a FRT and let $C$ be a cube. We denote by $\{e_i\}_{i=1}^{12}$ a finite sequence such that $\{e_i|i=1,...,12\}$ is equal to the set of all edges of $C$. Suppose $A \subset C$ and there exist points $a_1,...,a_{12}$ of $A$ such that $a_i \in e_i$ for each $i=1,...,12$. Then for each $j\in \{1,2,3,4\}$, there exists a vertex $p_j$ of $C$ such that edges $\{e_i|a_i \in A_j \}$ of $C$ share $p_j$ and $p_1, p_2, p_3, p_4$ are vertices of a regular tetrahedron.
\end{lem}
\begin{prf*}
First, we prove the following claim.

\underline{Claim1} Let $A$ be a FRT, let $C$ be a cube such that $A\subset C$ and let $a,b$ be points of $A$. If there exist two faces $S$ and $T$ of $C$ which are parallel to each other such that $a\in S$ and $b\in T$, then there exist $i\neq j \in \{1,2,3,4\}$ such that $a\in A_i$ and $b\in A_j$.  

To prove Claim1, we define $u$ as a vector which is normal to $S$ and $T$ as in Figure \ref{aaa2}. There exists $i\in \{1,2,3,4\}$ such that $a\in A_i$. We have $\{a+(v_l-v_{i})|l\neq{i}\}\subset{A}(\subset{C})$ by Proposition $\ref{trns}$. Furthermore $\{\overline{v_{i}v_l}\}_{l\neq{i}}$ are edges of regular tetrahedron $T_0$, thus there exists $l\in \{1,2,3,4\}$ with $l\neq{i}$ such that $a+(v_l-v_{i})$ belongs to the negative side of $S$ with respect to $u$. If $b\in A_i$, then $b+(v_l-v_i)$ belongs to the negative side of $T$ since $S$ and $T$ are facing each other. But this contradicts $b+(v_l-v_i)\in A_l\subset C$. Hence there exists $j\in \{1,2,3,4\}$ with $j\neq{i}$ such that $b\in A_j$. We have thus proved the claim.

Suppose that there exist mutually distinct numbers $i_1, i_2, i_3, i_4\in \{1,2,...,12\}$ and $j\in \{1,2,3,4\}$ such that for each $k\in\{i_1, i_2, i_3, i_4\}$, $a_k\in A_j$. Then there exist two distinct numbers $k,l\in \{i_1, i_2, i_3, i_4\}$ such that two edges $e_{k}$ and $e_{l}$ are parallel to each other or twisted, and hence $a_{k},a_{l}\in A_j$. But this contradicts Claim1. Hence for each $j\in \{	1,2,3,4\}$, we have that the cardinality of $\{k\in \{1,2,...,12\}| a_k\in A_j \}$ is equal to $3$. For each $j\in \{1,2,3,4\}$, let $k_1, k_2, k_3$ be mutually distinct elements in $\{k\in \{1,2,...,12\}| a_k\in A_j \}$. Suppose that the edges $e_{k_1},e_{k_2}$ and $e_{k_3}$ do not share a vertex of $C$, then there exist two distinct numbers $k,l\in \{k_1,k_2,k_3\}$ such that two edges $e_{k}$ and $e_{l}$ are parallel to each other or twisted, and hence $a_{k},a_{l}\in A_j$. But this contradicts Claim1. Hence the edges $e_{k_1},e_{k_2}$ and $e_{k_3}$ share a vertex of $C$.         

Then we have proved our lemma. \qed
\end{prf*}
\begin{figure}[h]
  \begin{center}
    \begin{tabular}{c}

      \begin{minipage}{0.4\hsize}
        \begin{center}
          \includegraphics[clip, width=5cm]{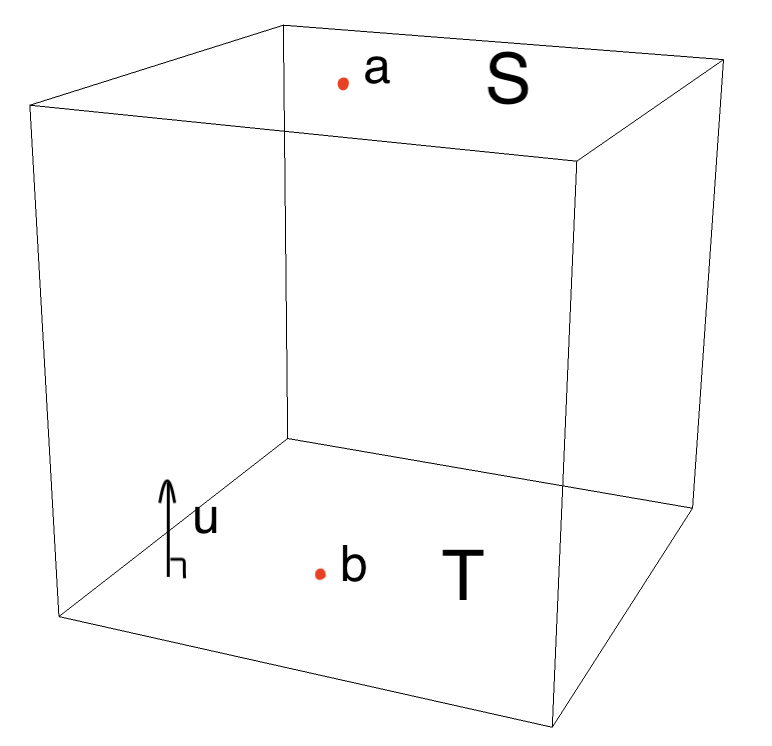}
         \caption{The face $S$ is parallel to $T$. The vector $u$ is normal to $S$ and $T$.}
     \label{aaa2}

        \end{center}
      \end{minipage}

    \end{tabular}

  \end{center}
\end{figure}
As a corollary of Lemma \ref{ho1}, we get Lemma \ref{ho2} and Lemma \ref{ho3}.
\begin{lem}
\label{ho2}
Let $A=A(c,P)$ be a FRT and let $C$ be a cube. We denote by $\{e_i\}_{i=1}^{12}$ a finite sequence such that $\{e_i|i=1,...,12\}$ is equal to the set of all edges of $C$. Suppose $A \subset C$ and there exist points $a_1,...,a_{12}$ of $A$ such that $a_i \in e_i$ for each $i=1,...,12$. Then $C$ is a normal cube.
\end{lem}
\begin{prf*}
By Lemma \ref{ho1}, there exist vertices of a regular tetrahedron $p_1, p_2, p_3, p_4$ of $C$ such that edges $\{e_i|a_i \in A_j \}$ of $C$ share $p_j$ for each $j\in\{1,2,3,4\}$(without loss of generality, we can put the points $p_1, p_2, p_3, p_4$ as in Figure \ref{ccc}). 

For each $i\in \{2,3,4\}$, let $S_i$ be the face which contain $p_1$ and $p_i$(See Figure \ref{ccc1}). 

Let $e_2=S_2\cap S_4$, $e_3=S_2\cap S_3$ and $e_4=S_3\cap S_4$(See Figure \ref{ccc2}).

Since the definition of $p_1$, there exist points $a_2\in e_2\cap A_1$, $a_3\in e_3\cap A_1$ and $a_4\in e_4\cap A_1$(See Figure \ref{ccc3}).

For each $i\in \{2,3,4\}$, let $e^{\prime}_i$ be the edge of $S_i$ which is parallel to $e_i$(See Figure \ref{ccc4}).

Since the definition of $p_i$, there exists points $\alpha_i\in e^{\prime}_i\cap A_i$ for each $i\in \{2,3,4\}$(See Figure \ref{ccc5}).

Let $u_2,u_3,u_4$ be three orthogonal vectors such that for each $i=2,3,4$, the initial point is $p_1$ and $u_i$ is parallel to $e_i$(See Figure \ref{ccc6}).

We prove that the edge $\overline{v_1v_i}$ is parallel to $S_i$ for all $i\in\{2,3,4\}$. Suppose that $\overline{v_1v_2}$ is not parallel to $S_2$. Then we have $a_2+(v_2-v_1)$ or $\alpha_2+(v_1-v_2)$ belongs to the negative side of $S_2$ with respect to $u_4$. But this contradicts $a_2+(v_2-v_1)\in A_2\subset C$ or $\alpha_2+(v_1-v_2)\in A_1\subset C$. Suppose that $\overline{v_1v_3}$ is not parallel to $S_3$. Then we have $a_3+(v_3-v_1)$ or $\alpha_3+(v_1-v_3)$ belongs to the negative side of $S_3$ with respect to $u_2$. But this contradicts $a_3+(v_3-v_1)\in A_3\subset C$ or $\alpha_3+(v_1-v_3)\in A_1\subset C$. Suppose that $\overline{v_1v_4}$ is not parallel to $S_4$. Then we have $a_4+(v_4-v_1)$ or $\alpha_4+(v_1-v_4)$ belongs to the negative side of $S_4$ with respect to $u_3$. But this contradicts $a_4+(v_4-v_1)\in A_4\subset C$ or $\alpha_4+(v_1-v_4)\in A_1\subset C$. Hence $S_2, S_3, S_4$ are parallel to $\overline{v_1v_2},\overline{v_1v_3},\overline{v_1v_4}$ respectively.

Since $\overline{v_1v_2},\overline{v_1v_3},\overline{v_1v_4}$ are edges of $T_0$ which share $v_1$ and $C_0$ is a cube whose vertices contains $v_1, v_2, v_3, v_4$, $C$ is a normal cube(See Figure \ref{ccc7}).

Then we have proved our lemma. \qed
\end{prf*}

\begin{figure}[h]
  \begin{center}
    \begin{tabular}{c}

      \begin{minipage}{0.4\hsize}
        \begin{center}
          \includegraphics[clip, width=5cm]{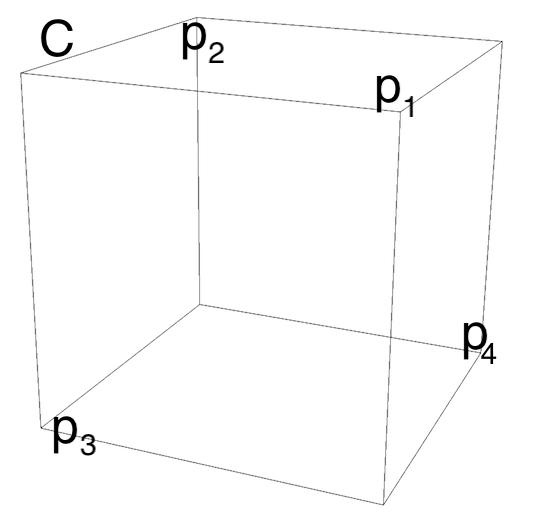}
         \caption{}
     \label{ccc}

        \end{center}
      \end{minipage}
\hspace{0.3cm}
   
      \begin{minipage}{0.4\hsize}
        \begin{center}
          \includegraphics[clip, width=5cm]{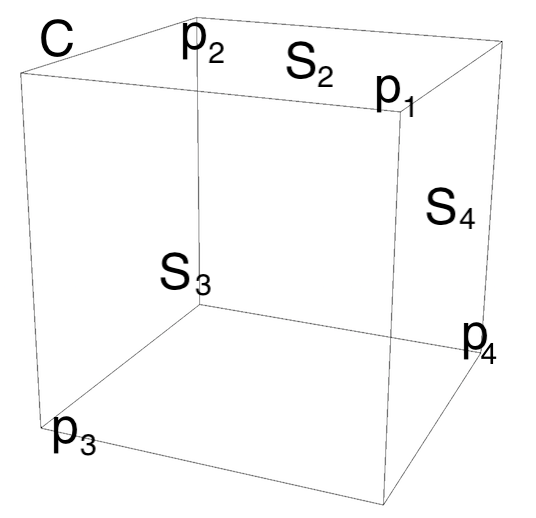}
         \caption{}
       \label{ccc1}
    
        \end{center}
      \end{minipage}

    \end{tabular}

  \end{center}
\end{figure}

\begin{figure}[h]
  \begin{center}
    \begin{tabular}{c}

      \begin{minipage}{0.4\hsize}
        \begin{center}
          \includegraphics[clip, width=5cm]{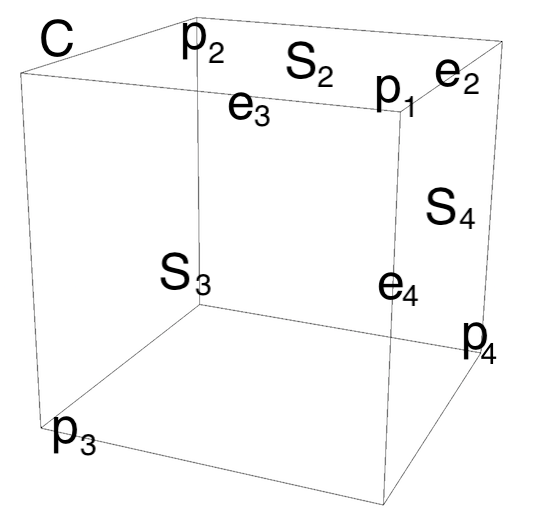}
         \caption{}
     \label{ccc2}

        \end{center}
      \end{minipage}
\hspace{0.3cm}
   
      \begin{minipage}{0.4\hsize}
        \begin{center}
          \includegraphics[clip, width=5cm]{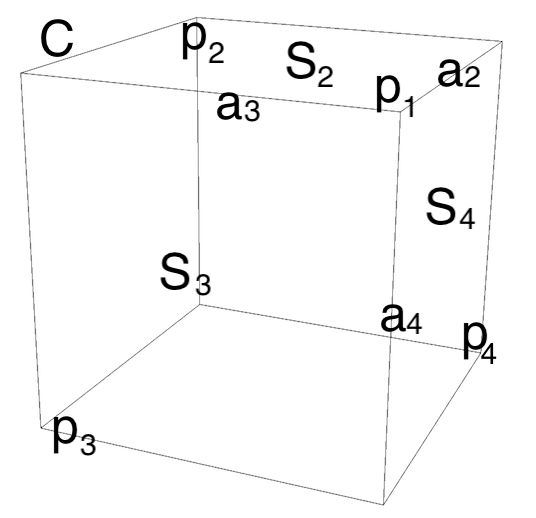}
         \caption{}
       \label{ccc3}
    
        \end{center}
      \end{minipage}

    \end{tabular}

  \end{center}
\end{figure}

\begin{figure}[h]
  \begin{center}
    \begin{tabular}{c}

      \begin{minipage}{0.4\hsize}
        \begin{center}
          \includegraphics[clip, width=5cm]{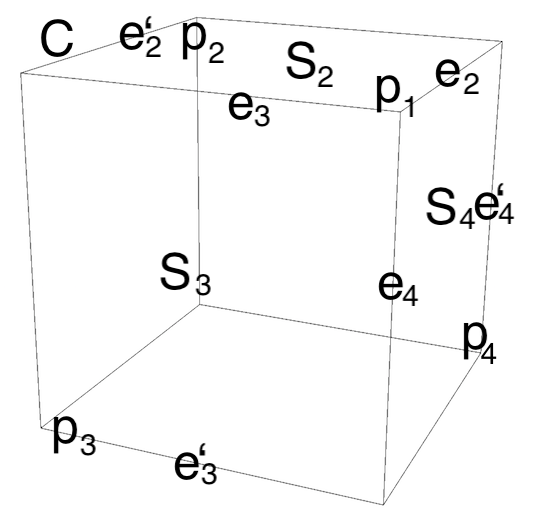}
         \caption{}
     \label{ccc4}

        \end{center}
      \end{minipage}
\hspace{0.3cm}
   
      \begin{minipage}{0.4\hsize}
        \begin{center}
          \includegraphics[clip, width=5cm]{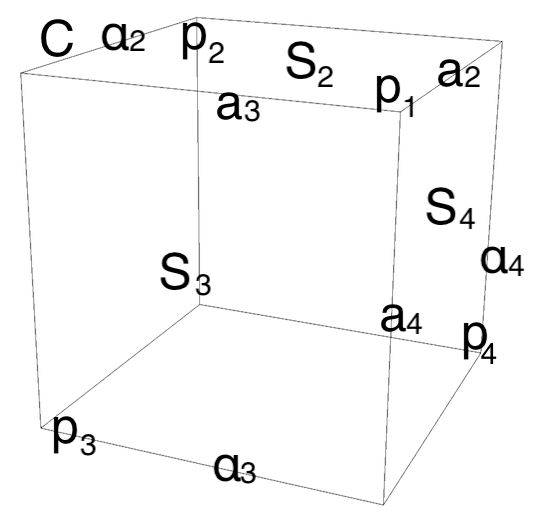}
         \caption{}
       \label{ccc5}
    
        \end{center}
      \end{minipage}

    \end{tabular}

  \end{center}
\end{figure}
\begin{figure}[h]
  \begin{center}
    \begin{tabular}{c}

      \begin{minipage}{0.4\hsize}
        \begin{center}
          \includegraphics[clip, width=5cm]{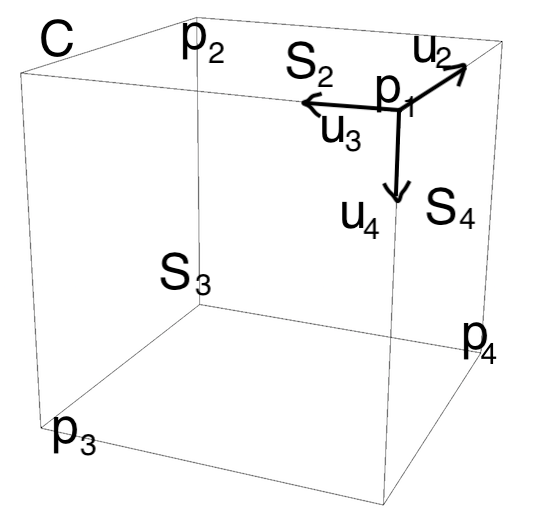}
         \caption{}
     \label{ccc6}

        \end{center}
      \end{minipage}
\hspace{0.3cm}
   
      \begin{minipage}{0.4\hsize}
        \begin{center}
          \includegraphics[clip, width=5cm]{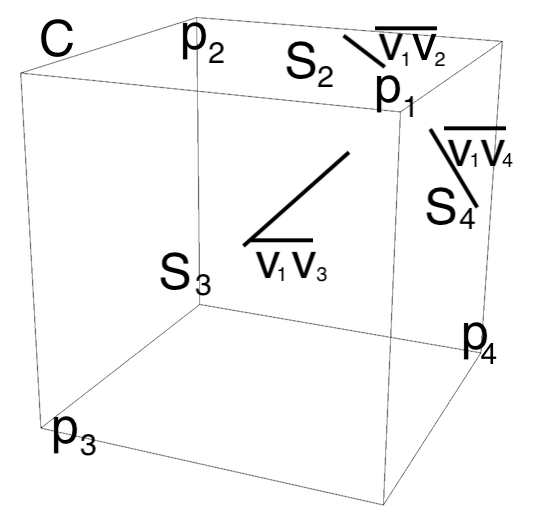}
         \caption{}
       \label{ccc7}
    
        \end{center}
      \end{minipage}

    \end{tabular}

  \end{center}
\end{figure}

\begin{lem}
\label{ho3}
Let $A=A(c,P)$ be a FRT and let $C$ be a normal cube. Let $\{S_i\}_{i=1}^6$ be faces of $C$ and let $u_1,u_2,u_3$ be orthonormal vectors parallel to $x,y$ and $z$ axis respectively  as in Figure \ref{cuu}. Suppose $A \subset C$. Let $a_1,...,a_{12}$ be points of $A$ such that for each edge $e$ of $C$, there exists $i$ with $a_i\in e$ as in Figure \ref{cu}. Then $\{a_1,a_2,a_3\}\subset{A_1},\{a_4,a_5,a_6\}\subset{A_2},\{a_7,a_8,a_9\}\subset{A_3}$ and $\{a_{10},a_{11},a_{12}\}\subset{A_4}$.
\end{lem}
\begin{prf*}
By Lemma \ref{ho1}, it follows that 
\begin{itemize}
\item[(i)]$\{a_1,a_2,a_3\}\subset{A_{i_1}},\{a_4,a_5,a_6\}\subset{A_{i_2}},\{a_7,a_8,a_9\}\subset{A_{i_3}},\{a_{10},a_{11},a_{12}\}\subset{A_{i_4}}$, or 
\item[(ii)] $\{a_1,a_7,a_{12}\}\subset{A_{i_1}},\{a_2,a_5,a_{10}\}\subset{A_{i_2}},\{a_3,a_6,a_9\}\subset{A_{i_3}}, \{a_4, a_8, a_{11} \}\subset{A_{i_4}}$
\end{itemize}
 Here, $i_1,i_2,i_3$ and $i_4$ are mutually distinct. First, we show that (ii) never happen.

Suppose that $i_1=1$. Since $C$ is normal, it follows that $a_7+(v_3-v_1)$ belongs to the negative side of $S_4$ with respect to $u_3$. But this contradicts that $a_7+(v_3-v_1)\in A_3 \subset C$. 

Suppose that $i_1=2$. Then it follows that $a_1+(v_1-v_2)$ belongs to the positive side of $S_2$ with respect to $u_1$. But this contradicts that $a_1+(v_1-v_2)\in A_1 \subset C$.

Suppose that $i_1=3$. Then it follows that $a_1+(v_1-v_3)$ belongs to the positive side of $S_2$ with respect to $u_1$. But this contradicts that $a_1+(v_1-v_2)\in A_1 \subset C$.

Suppose that $i_1=4$. Then it follows that $a_1+(v_1-v_4)$ belongs to the negative side of $S_1$ with respect to $u_2$. But this contradicts that $a_1+(v_1-v_4)\in A_1 \subset C$.

Hence (i) holds. 

We next show that 
\begin{align}
i_1=1, i_2=2, i_3=3, i_4=4.
\end{align}

In order to prove $i_1=1$, suppose that $i_1=2$. Then it follows that $a_1+(v_1-v_2)$ belongs to the positive side of $S_2$ with respect to $u_1$. But this contradicts that $a_1+(v_1-v_2)\in A_1 \subset C$.

Suppose that $i_1=3$. Then it follows that $a_1+(v_1-v_3)$ belongs to the positive side of $S_2$ with respect to $u_1$. But this contradicts that $a_1+(v_1-v_3)\in A_1 \subset C$.

Suppose that $i_1=4$. Then it follows that $a_1+(v_1-v_4)$ belongs to the negative side of $S_1$ with respect to $u_2$. But this contradicts that $a_1+(v_1-v_4)\in A_1 \subset C$.

Hence it follows that $i_1=1$.

In order to prove $i_2=2$, suppose that $i_2=3$. Then it follows that $a_5+(v_2-v_3)$ belongs to the positive side of $S_3$ with respect to $u_3$. But this contradicts that $a_5+(v_2-v_3)\in A_2 \subset C$.

Suppose that $i_2=4$. Then it follows that $a_5+(v_2-v_4)$ belongs to the positive side of $S_3$ with respect to $u_3$. But this contradicts that $a_5+(v_2-v_4)\in A_2 \subset C$.

Hence it follows that $i_2=2$.

In order to prove $i_3=3$, suppose that $i_3=4$. Then it follows that $a_7+(v_3-v_4)$ belongs to the negative side of $S_1$ with respect to $u_2$. But this contradicts that $a_7+(v_3-v_4)\in A_3 \subset C$.

Hence it follows that $i_3=3$.

Since $i_1,i_2,i_3$ and $i_4$ are mutually distinct, $i_4=4$. Hence we have proved (1). 

Thus we have proved our lemma. \qed
\end{prf*}
\begin{figure}[h]
  \begin{center}
    \begin{tabular}{c}

      \begin{minipage}{0.4\hsize}
        \begin{center}
          \includegraphics[clip, width=5cm]{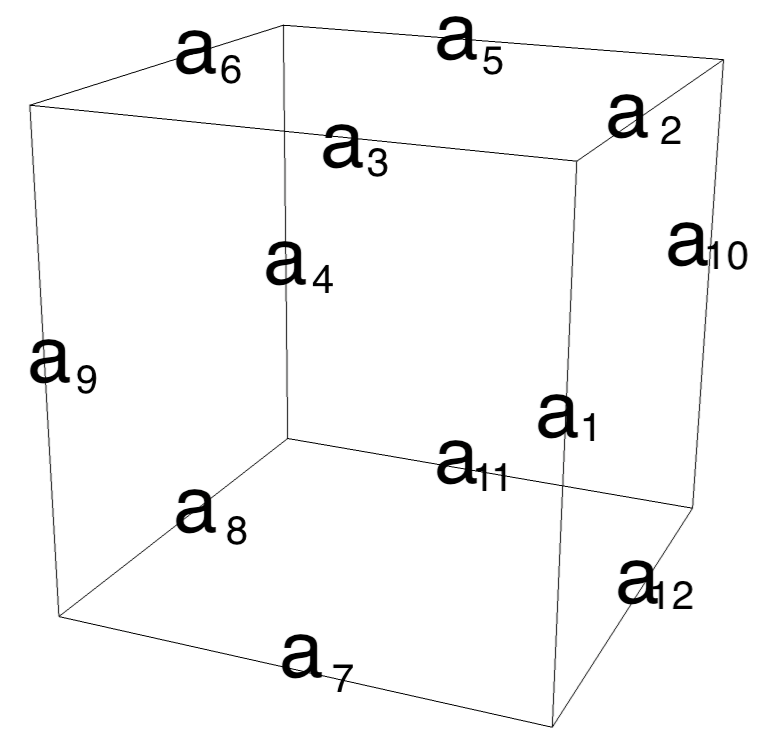}
         \caption{}
     \label{cu}

        \end{center}
      \end{minipage}

  \hspace{0.3cm} 
      \begin{minipage}{0.4\hsize}
        \begin{center}
          \includegraphics[clip, width=5cm]{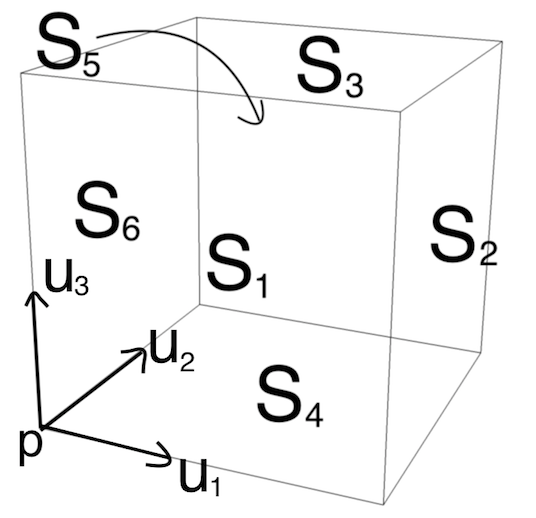}
         \caption{The faces $S_1$, $S_4$ and $S_6$ share the vertex $p$. Each face $S_1$, $S_4$ and $S_6$ is paralleled to $S_5$, $S_3$ and $S_2$ respectively.}
       \label{cuu}
    
        \end{center}
      \end{minipage}

    \end{tabular}

  \end{center}
\end{figure}
\begin{figure}[h]
  \begin{center}
    \begin{tabular}{c}

      \begin{minipage}{0.4\hsize}
        \begin{center}
          \includegraphics[clip, width=5cm]{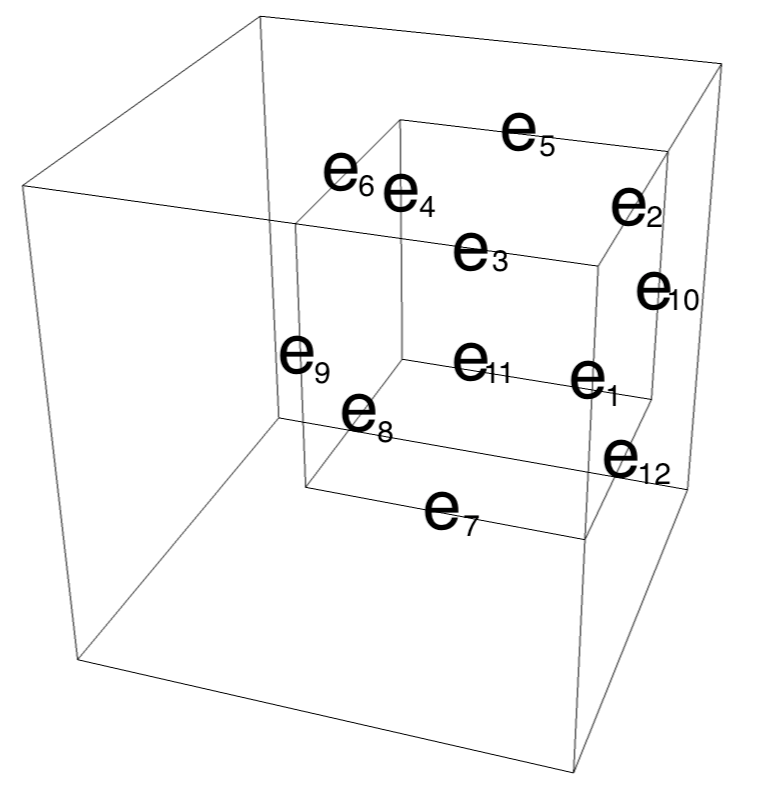}
         \caption{The edges $e_1$, $e_2$ and $e_3$ share a vertex, the edges $e_4$, $e_5$ and $e_6$ share a vertex, the edges $e_7$, $e_8$ and $e_9$ share a vertex and the edges $e_{10}$, $e_{11}$ and $e_{12}$ share a vertex.}
     \label{g1}

        \end{center}
      \end{minipage}
\hspace{0.3cm}
   
      \begin{minipage}{0.4\hsize}
        \begin{center}
          \includegraphics[clip, width=5cm]{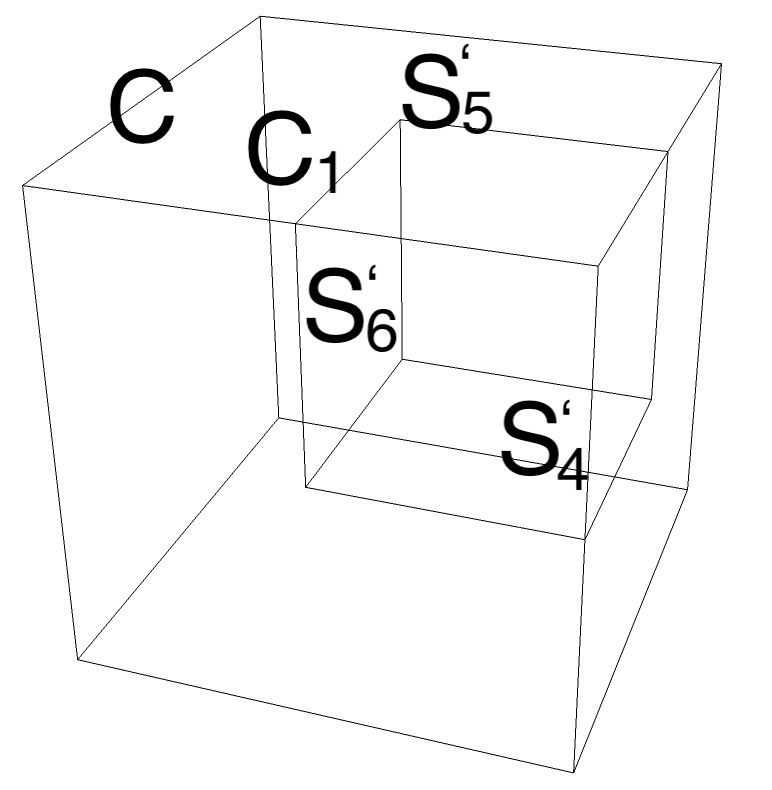}
         \caption{The bigger cube is $C$ and the smaller cube is $C_1$.}
       \label{g2}
    
        \end{center}
      \end{minipage}

    \end{tabular}

  \end{center}
\end{figure}
We now give the proof of Proposition \ref{decide3}.
\begin{pf*}[proof of Proposition \ref{decide3}]
Fix $c\in (0,1), P\in{SO(3)}\backslash{\mathcal{P}_6}$. The proof is done by contradiction. Suppose $A(c,P)$ is an IC of a cube $C$. By Remark \ref{rem}, $A\subset{C}$ and there exists a point of ${A}$ on each edge of $C$. By Lemma {\ref{ho2}}, $C$ is a normal cube. Let $a_1,...,a_{12}$ be points of $A$ such that for each edge $e$ of $C$, there exists $i$ with $a_i\in e$(See Figure \ref{cu}). Let $S_1,...,S_6$ be faces of $C$ and let $u_1,u_2,u_3$ orthonormal vectors parallel to $x,y$ and $z$ axis respectively as in Figure \ref{cuu}. By Lemma \ref{ho3}, it follows that $\{a_1,a_2,a_3\}\subset{A_1},\{a_4,a_5,a_6\}\subset{A_2},\{a_7,a_8,a_9\}\subset{A_3}$ and $\{a_{10},a_{11},a_{12}\}\subset{A_4}$.

To prove Proposition \ref{decide3}, we use the following claim.

\underline{Claim$2$} Let $C_1$ be the cube which is enclosed by $S_1, S_2, S_3, {S_4}^{\prime}=S_4+v_1-v_3, {S_5}^{\prime}=S_5+v_1-v_2, {S_6}^{\prime}=S_6+v_1-v_2$. Then $A_1=f_1(A)\subset C_1$ and there exists a point of $A_1$ on each edge of $C_1$ .

We show Claim2. Note that $C_1$ is well-defined since $a_7+v_1-v_3\in A_1\subset C$, $a_5+v_1-v_2\in A_1\subset C$ and $a_6+v_1-v_2\in A_1\subset C$. Note that vectors $v_1-v_2,v_1-v_3,v_1-v_4$ are parallel to $S_3,S_1,S_2$ respectively and the cube which is enclosed by $S_1, S_2, S_3, {S_4}^{\prime}, {S_5}^{\prime}, {S_6}^{\prime}$ is equal to the cube which is enclosed by $S_1, S_2, S_3, {S_4}^{\prime\prime}=S_4+v_1-v_4,{S_5}^{\prime},{S_6}^{\prime}$.

We denote by $\{e_i\}_{i=1}^{12}$ the edges of $C_1$ as in Figures \ref{g1}, \ref{g2}.

For each $v\in \mR$, we denote by $x(v), y(v)$ and $z(v)$  first, second and third coordinate of $v$ respectively. For example, $x(v_1)=1, y(v_2)=-1$. Since $\{a_4,a_5,a_6\}\subset{A_2}$, $\{a_4,a_5,a_6\}+v_1-v_2\subset{A_1}$.

Since $a_6\in A_2\subset C,$ we have that $a_6+v_1-v_2\in{A_1}\subset C.$ Hence we have that 
\begin{align}
a_6+v_1-v_2\in{e_6}.
\end{align} 

Similarly, we have that 
\begin{align}
a_5+v_1-v_2\in{e_5}.
\end{align}

Similarly, since $\{a_7,a_8,a_9\}+v_1-v_3\subset{A_1}$, $\{a_{10},a_{11},a_{12}\}+v_1-v_4\subset{A_1}$, we have that 
\begin{align}
a_7+v_1-v_3\in{e_7}, a_9+v_1-v_3\in{e_9}, a_{10}+v_1-v_4\in{e_{10}}, a_{12}+v_1-v_4\in{e_{12}}.
\end{align}

Next, we prove that 
\begin{align}
a_4+v_1-v_2\in{e_4}.
\end{align} 
In order to prove (5), suppose that $a_4+v_1-v_2\notin{e_4}$. We have that $a_4+v_1-v_2$ belongs to the negative side of ${S_4}^{\prime}$ with respect to $u_3$. By (4), we have that $z(a_4+v_1-v_2)<z(a_{12}+v_1-v_4)$. Hence we have that $z(a_4+v_4-v_2)<z(a_{12})$, but this contradicts $a_4+v_4-v_2\in{A_4}\subset{C}$. Hence we have proved (5).

Similarly, we have that 
\begin{align}
a_8+v_1-v_3\in{e_8}.
\end{align} 
For, suppose that $a_8+v_1-v_3\notin{e_8}$. We have that $a_8+v_1-v_3$ belongs to the positive side of ${S_5}^{\prime}$ with respect to $u_2$. By (4), we have that $y(a_8+v_1-v_3)>y(a_{10}+v_1-v_4)$. Hence we have that $y(a_8+v_4-v_3)>y(a_{10})$, but this contradicts $a_8+v_4-v_3\in{A_4}\subset{C}$. Hence we have proved (6).

Similarly, we have that 
\begin{align}
a_{11}+v_1-v_4\in{e_{11}}.
\end{align} 
For, suppose that $a_{11}+v_1-v_4\notin{e_{11}}$. We have that $a_{11}+v_1-v_4$ belongs to the negative side of ${S_6}^{\prime}$ with respect to $u_1$. By (4), $x(a_{11}+v_1-v_4)<x(a_9+v_1-v_3)$. Hence we have that $x(a_{11}+v_3-v_4)<x(a_{9})$, but this contradicts $a_{11}+v_3-v_4\in{A_3}\subset{C}$. Hence we have proved (7).

Next, we prove that 
\begin{align}
a_1\in{e_1}. 
\end{align}
In order to prove (8), suppose that $a_1\notin{e_1}$. We have that $a_1$ belongs to the negative side of ${S_4}^{\prime}$ with respect to $u_3$. By (4),  $z(a_1)<z(a_7+v_1-v_3)$. Hence we have that $z(a_1+v_3-v_1)<z(a_7)$, but this contradicts $a_1+v_3-v_1\in{A_3}\subset{C}$. Hence we have proved (8).

Similarly, we have that 
\begin{align}
a_2\in{e_2}.
\end{align}
For, suppose that $a_2\notin{e_2}$. We have that $a_2$ belongs to the positive side of ${S_5}^{\prime}$ with respect to $u_2$. By (3), $y(a_2)>y(a_5+v_1-v_2)$. Hence we have that $y(a_2+v_2-v_1)>y(a_5)$, but this contradicts $a_2+v_2-v_1\in{A_2}\subset{C}$. Hence we have proved (9).

Similarly, we have that 
\begin{align}
a_3\in{e_3}.
\end{align}
For, suppose that $a_3\notin{e_3}$. We have that $a_3$ belongs to the negative side of ${S_6}^{\prime}$ with respect to $u_1$. By (2), $x(a_3)<x(a_6+v_1-v_2)$. Hence we have that $x(a_3+v_2-v_1)<x(a_6)$, but this contradicts $a_3+v_2-v_1\in{A_2}\subset{C}$. Hence we have proved (10).

Hence there exists a point of $A_1$ on each edge of $C_1$.

Finally we show $A_1\subset C_1$. Suppose that there exists a point $a\in A_1$ such that $a\notin C_1$. Since $A_1\subset C$, the point $a$ belongs to 
\begin{itemize}
\item[(I)] the negative side of ${S_4}^{\prime}=S_4+v_1-v_3$ with respect to $u_3$, or

\item[(II)] the positive side of ${S_5}^{\prime}=S_5+v_1-v_2$ with respect to $u_2$, or

\item[(III)] the negative side of ${S_6}^{\prime}=S_6+v_1-v_2$ with respect to $u_1$.
\end{itemize}
We now consider case (I). By (4), we have that $z(a_7+v_1-v_3)>z(a)$. Hence $z(a_7)>z(a+v_3-v_1)$. But this contradicts $a+v_3-v_1\in{A_3}\subset{C}$.

We next consider case (II). By (4), we have that $y(a_{10}+v_1-v_4)<y(a)$. Hence $y(a_{10})<y(a+v_4-v_1)$. But this contradicts $a+v_4-v_1\in{A_4}\subset{C}$. 

We finally consider case (III). By (2), we have that $x(a_6+v_1-v_2)>x(a)$. Hence $x(a_6)>x(a+v_2-v_1)$. But this contradicts $a+v_2-v_1\in{A_2}\subset{C}$.  Hence we have that $A_1\subset C_1$.

This completes the proof of Claim2. 

By Claim2, it follows that $A={f_1}^{-1}(A_1)\subset {f_1}^{-1}(C_1)$ and there exists a point of $A$ on each edge of ${f_1}^{-1}(C_1)$ . By Lemma \ref{ho2}, ${f_1}^{-1}(C_1)$ is a normal cube. But this contradicts $P\in{SO(3)}\backslash{\mathcal{P}_6}$. Then we have proved our proposition.\qed
\end{pf*}

\section{The properties of FRTs which are imaginary cubes}
We consider FRTs which are imaginary cubes. That is, we consider $A(c,P)$, where $c\in [1/2,1)$ and $P\in \mpo$ (see Theorem \ref{decide}). 

We set $V=\{v_1,v_2,v_3,v_4\}$ and $V^{\prime}=\{v^{\prime}_1,v^{\prime}_2,v^{\prime}_3,v^{\prime}_4\}$. Here, $v^{\prime}_1=(1,1,1), v^{\prime}_2=(-1,-1,1), v^{\prime}_3=(1,-1,-1), v^{\prime}_4=(-1,1,-1)$. Note that the set of vertices of $C_0$ is equal to 
$\{v_1, v_2, v_3, v_4, v^{\prime}_1, v^{\prime}_2, v^{\prime}_3, v^{\prime}_4\}$. We set $C_c=1/({1-c})C_0$, that is, $C_c$ is the cube whose vertices is equal to $\{v_1/(1-c),v_2/(1-c),v_3/(1-c),v_4/(1-c),v^{\prime}_1/(1-c),v^{\prime}_2/(1-c),v^{\prime}_3/(1-c),v^{\prime}_4/(1-c)\}$. There is a proposition about the cardinality of $\{A(c,P)| P \in \mathcal{P}_6 \}$ for each $c\in(0,1)$.
\begin{pro}
For each $c \in (0,1)$, the cardinality of $\{A(c,P)| P \in \mathcal{P}_6 \}$ is 2.
\end{pro}
\begin{proof}
Fix $c\in (0,1)$, $P_1\in \mathcal{P}_4$ and $P_2\in \mathcal{P}_6\backslash \mathcal{P}_4$. Since $P_1$ preserves $T_0$, $P_1V=V$. Since $P_2$ preserves $C_0$ but does not preserve $T_0$, we have $P_2V=V^{\prime}, P_2V^{\prime}=V$ and $P_2^{2n} \in \mathcal{P}_4$ for all $n$. Hence we get the following equations. 

\begin{align*}A(c,P_1)
&=\{\sum_{i=0}^{\infty}(cP_1)^{i}v_{\omega_{i+1}}|\omega=\omega_1\omega_2\cdots\in{\{1,2,3,4\}^{\infty}}\}\\
&=\{\sum_{i=0}^{\infty}c^{i}v_{\tau_{i+1}}|\tau=\tau_1\tau_2\cdots\in{\{1,2,3,4\}^{\infty}}\},\\
A(c,P_2)
&=\{\sum_{i=0}^{\infty}(cP_2)^{i}v_{\omega_{i+1}}|\omega=\omega_1\omega_2\cdots\in{\{1,2,3,4\}^{\infty}}\}\\
&=\{\sum_{i=0}^{\infty}(cP_2)^{2i}v_{\omega_{2i+1}}+\sum_{i=0}^{\infty}(cP)^{2i+1}v_{\omega_{2(i+1)}}|\omega=\omega_1\omega_2\cdot\cdot\cdot\in{\{1,2,3,4\}^{\infty}}\}\\
&=\{\sum_{i=0}^{\infty}c^{2i}v_{\tau_{2i+1}}+\sum_{i=0}^{\infty}c^{2i+1}v^{\prime}_{\tau_{2(i+1)}}|\tau=\tau_1\tau_2\cdots\in{\{1,2,3,4\}^{\infty}}\}.
\end{align*}
Hence $A(c,P_1)$ and $A(c,P_2)$ are not independent of choice of $P_1\in \mathcal{P}_4$ and $P_2\in \mathcal{P}_6\backslash \mathcal{P}_4$ respectively. Then $A(c,P_1)$ has the point $v_1/(1-c)$. For, if we set $\tau=\overline{1}$, $\sum_{i=0}^{\infty}c^{i}v_{\tau_{i+1}}=v_1/(1-c)$. Furthermore, $A(c,P_2)$ does not have the point $v_1/(1-c)$. To prove this, suppose that $A(c,P_2)$ has the point $v_1/(1-c)$. Then there exists $\tau=\tau_1\tau_2\cdots\in{\{1,2,3,4\}^{\infty}}$ such that $\sum_{i=0}^{\infty}c^{2i}v_{\tau_{2i+1}}+\sum_{i=0}^{\infty}c^{2i+1}v^{\prime}_{\tau_{2(i+1)}}=v_1/(1-c)$. For $v\in \mR$, we denote by $x(v), y(v)$ and $z(v)$  the first, second and third coordinate of $v$ respectively. Then 
\begin{align*}
1/(1-c)&=x(v_1/(1-c))\\&=x(\sum_{i=0}^{\infty}c^{2i}v_{\tau_{2i+1}}+\sum_{i=0}^{\infty}c^{2i+1}v^{\prime}_{\tau_{2(i+1)}})\\&=x(\sum_{i=0}^{\infty}c^{2i}v_{\tau_{2i+1}})+x(\sum_{i=0}^{\infty}c^{2i+1}v^{\prime}_{\tau_{2(i+1)}})\\&=\sum_{i=0}^{\infty}c^{2i}x(v_{\tau_{2i+1}})+\sum_{i=0}^{\infty}c^{2i+1}x(v^{\prime}_{\tau_{2(i+1)}})
\end{align*}
If there exists $i$ such that $x(v_{\tau_{2i+1}})=-1$ or $x(v^{\prime}_{\tau_{2i+1}})=-1$, then $\sum_{i=0}^{\infty}c^{2i}x(v_{\tau_{2i+1}})+\sum_{i=0}^{\infty}c^{2i+1}x(v^{\prime}_{\tau_{2(i+1)}})<1/(1-c)$. Then we have that for each $i$, $v_{\tau_{2i+1}}=v_1$ or $v_3$ and $v^{\prime}_{\tau_{2i+1}}=v^{\prime}_1$ or $v^{\prime}_3$.

Similarly, if we consider $y(\sum_{i=0}^{\infty}c^{2i}v_{\tau_{2i+1}}+\sum_{i=0}^{\infty}c^{2i+1}v^{\prime}_{\tau_{2(i+1)}})$, then we have that for each $i$, $v_{\tau_{2i+1}}=v_1$ and $v^{\prime}_{\tau_{2i+1}}=v^{\prime}_3$.

Then
\begin{align*}
1/(1-c)&=z(v_1/(1-c))\\&=z(\sum_{i=0}^{\infty}c^{2i}v_{\tau_{2i+1}}+\sum_{i=0}^{\infty}c^{2i+1}v^{\prime}_{\tau_{2(i+1)}})\\&=z(\sum_{i=0}^{\infty}c^{2i}v_{1}+\sum_{i=0}^{\infty}c^{2i+1}v^{\prime}_{3})\\&=\sum_{i=0}^{\infty}c^{2i}1+\sum_{i=0}^{\infty}c^{2i+1}(-1)\\&<1/(1-c)
\end{align*}
But this is a contradiction. Hence $A(c,P_2)$ does not have the point $v_1/(1-c)$. Hence $A(c,P_1)\neq A(c,P_2)$.
%
%
Then we have proved our remark.
\end{proof}
For each $c$, 
we denote by $T(c)$ the FRT $A(c,P_1)$, where $P_1\in \mpt$ and we denote by $O(c)$ the FRT $A(c,P_2)$, where $P_2\in \mpo \backslash \mpt$(See Figures \ref{H01}, \ref{H91}). 
\begin{figure}[h]
  \begin{center}
    \begin{tabular}{c}

      \begin{minipage}{0.4\hsize}
        \begin{center}
          \includegraphics[clip, width=5cm]{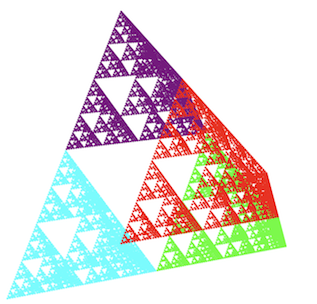}
         \caption{$T(1/2)$}
     \label{H01}

        \end{center}
      \end{minipage}
\hspace{0.3cm}
   
      \begin{minipage}{0.4\hsize}
        \begin{center}
          \includegraphics[clip, width=5cm]{H91.png}
         \caption{$O(1/2)$}
       \label{H91}
    
        \end{center}
      \end{minipage}

    \end{tabular}

  \end{center}
\end{figure}
To consider the convex hulls of $T(c)$ or $O(c)$, we use the following lemmas. Here, for each $A\subset \mR$, $\overline{A}$ denotes the closure of $A$ with respect to Euclidean topology and the convex hull of $A$ is defined as the set $\overline{\bigcup_{j=1}^{\infty}\bigcup_{p_1,p_2,...,p_j\in A}\{\sum_{i=1}^jq_ip_i|0\le q_1,q_2,...,q_j, \sum_{i=1}^jq_i=1\}}$. We denote by $co(A)$ the convex hull of $A$.
\begin{lem}
\label{convex}
Let $\varphi_1,\varphi_2,...,\varphi_n$ be contracting mappings on $\mR$ into $\mR$ defined by $\varphi_i(x)=c(x-p_i)+p_i$, where $p_i\in \mR$ and $c\in(0,1)$ for each $i=1,2,...,n$. Let $A$ be the attractor of $\{\varphi_1, \varphi_2,..., \varphi_n\}$. Then the convex hull of $A$ is equal to the convex hull of $\{p_1, p_2,..., p_n\}$. 
\end{lem}
\begin{proof}
We have that 
\begin{align*}
co(\{p_1,...,p_n\})=\overline{\bigcup_{j=1}^{\infty}\bigcup_{\omega_1\omega_2\cdots\omega_j\in \{1,2...,n\}^j}\{\sum_{i=1}^{j}q_{i}p_{\omega_i}|0\le q_1,q_2,...,q_j, \sum_{i=1}^jq_i=1\}}.
\end{align*}
Furthermore, it is well known that 
\begin{align*}
A=\overline{\bigcup_{j=1}^{\infty}\bigcup_{\omega_1\omega_2\cdots\omega_j\in \{1,2...,n\}^j}\{\varphi_{\omega_1}\varphi_{\omega_2}\cdots \varphi_{\omega_j}(v)}\},
\end{align*}

where $v\in A$.

Since $p_1, p_2,..., p_n\in A$, we have that 
\begin{align*}
A=\bigcup_{i=1}^n\overline{\bigcup_{j=1}^{\infty}\bigcup_{\omega_1\omega_2\cdots\omega_j\in \{1,2...,n\}^j}\{\varphi_{\omega_1}\varphi_{\omega_2}\cdots \varphi_{\omega_j}(p_i)}\}.
\end{align*}

For each $\omega_1\omega_2\cdots\omega_j\in \{1,2...,n\}^j,$ we have that 
\begin{align*}
\varphi_{\omega_1}\varphi_{\omega_2}\cdots \varphi_{\omega_j}(p_i)=c^jp_i+\sum_{l=0}^{j-1}c^l(1-c)p_{\omega_{l+1}}.
\end{align*}

Since for each $j$, $c^j+\sum_{l=0}^{j-1}c^l(1-c)=1$, we have that $\varphi_{\omega_1}\varphi_{\omega_2}\cdots \varphi_{\omega_j}(p_i)\in co(\{p_1,...,p_n\})$. Hence $A\subset co(\{p_1,...,p_n\})$. 

Since $p_1, p_2,..., p_n\in A$, the convex hull of $A$ $co(A)$ contains $co(\{p_1,...,p_n\})$.

Then we have proved our lemma.
\end{proof}
\begin{not*}
Let $I=\{1,2,3,4\}.$ For each $c$, we denote by $V(c)$ the set of points $\{(v_i+cv^{\prime}_j)/(1-c^2)\}_{(i,j)\in I^2}$, that is, $V(c)=$ 
\begin{align*}\{
&(\frac{1}{1-c},\frac{-1}{1+c},\frac{1}{1-c}),(\frac{1}{1+c},\frac{-1}{1-c},\frac{1}{1-c}),(\frac{1}{1-c},\frac{-1}{1-c},\frac{1}{1+c}),(\frac{1}{1+c},\frac{-1}{1+c},\frac{1}{1+c}),\\
&(\frac{-1}{1-c},\frac{1}{1+c},\frac{1}{1-c}),(\frac{-1}{1+c},\frac{1}{1-c},\frac{1}{1-c}),(\frac{-1}{1-c},\frac{1}{1-c},\frac{1}{1+c}),(\frac{-1}{1+c},\frac{1}{1+c},\frac{1}{1+c}),\\
&(\frac{-1}{1-c},\frac{-1}{1+c},\frac{-1}{1-c}),(\frac{-1}{1-c},\frac{-1}{1-c},\frac{-1}{1+c}),(\frac{-1}{1+c},\frac{-1}{1-c},\frac{-1}{1-c}),(\frac{-1}{1+c},\frac{-1}{1+c},\frac{-1}{1+c}),\\
&(\frac{1}{1-c},\frac{1}{1+c},\frac{-1}{1-c}),(\frac{1}{1+c},\frac{1}{1-c},\frac{-1}{1-c}),(\frac{1}{1-c},\frac{1}{1-c},\frac{-1}{1+c}),(\frac{1}{1+c},\frac{1}{1+c},\frac{-1}{1+c})\}.
\end{align*} 

We set 
\begin{align*}
&a_{c,1}=(\frac{1}{1+c},\frac{-1}{1+c},\frac{1}{1+c}), a_{c,2}=(\frac{-1}{1+c},\frac{1}{1+c},\frac{1}{1+c}),\\
& a_{c,3}=(\frac{-1}{1+c},\frac{-1}{1+c},\frac{-1}{1+c}),a_{c,4}=(\frac{1}{1+c},\frac{1}{1+c},\frac{-1}{1+c}).
\end{align*}
Let $C_c=1/(1-c)C_0$. We set $V^{\prime}(c)=V(c)\cap C_c(=V(c)\backslash\{a_{c,1},a_{c,2},a_{c,3},a_{c,4}\})$, that is, $V^{\prime}(c)=$ 
\begin{align*}\{
&(\frac{1}{1-c},\frac{-1}{1+c},\frac{1}{1-c}),(\frac{1}{1+c},\frac{-1}{1-c},\frac{1}{1-c}),(\frac{1}{1-c},\frac{-1}{1-c},\frac{1}{1+c}),\\
&(\frac{-1}{1-c},\frac{1}{1+c},\frac{1}{1-c}),(\frac{-1}{1+c},\frac{1}{1-c},\frac{1}{1-c}),(\frac{-1}{1-c},\frac{1}{1-c},\frac{1}{1+c}),\\
&(\frac{-1}{1-c},\frac{-1}{1+c},\frac{-1}{1-c}),(\frac{-1}{1-c},\frac{-1}{1-c},\frac{-1}{1+c}),(\frac{-1}{1+c},\frac{-1}{1-c},\frac{-1}{1-c}),\\
&(\frac{1}{1-c},\frac{1}{1+c},\frac{-1}{1-c}),(\frac{1}{1+c},\frac{1}{1-c},\frac{-1}{1-c}),(\frac{1}{1-c},\frac{1}{1-c},\frac{-1}{1+c})\}.
\end{align*} 
In Figure \ref{V}, the set of red points is equal to $V^{\prime}(c)$ and the set of blue points is equal to $\{a_{c,1},a_{c,2},a_{c,3},a_{c,4}\}$.
\end{not*}
\begin{figure}[h]
  \begin{center}
    \begin{tabular}{c}

      \begin{minipage}{0.4\hsize}
        \begin{center}
          \includegraphics[clip, width=5cm]{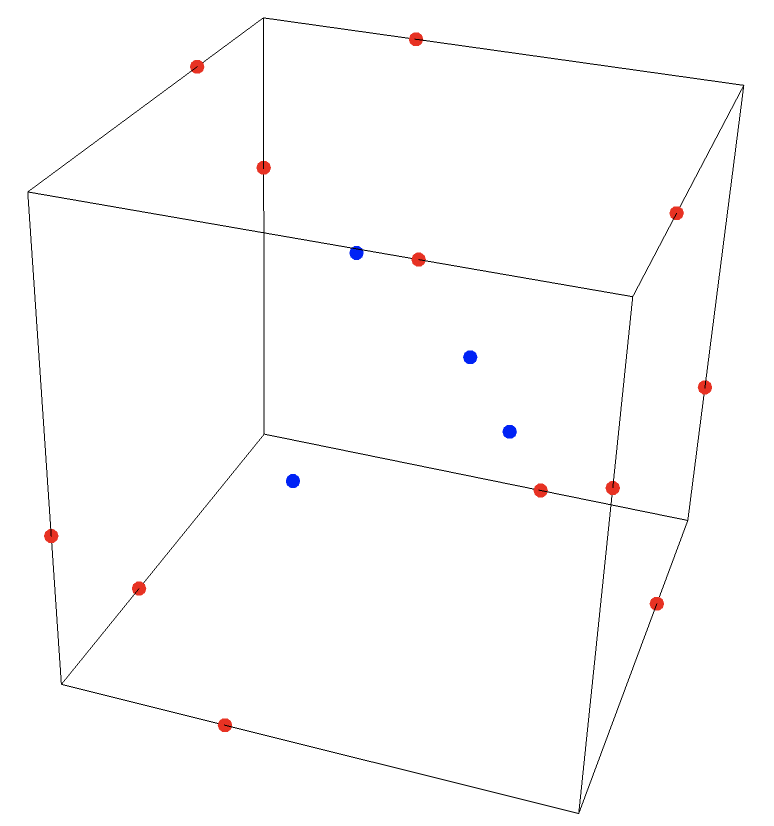}
         \caption{$V(1/2)$}
     \label{V}

        \end{center}
      \end{minipage}
\hspace{0.3cm}
   
      \begin{minipage}{0.4\hsize}
        \begin{center}
          \includegraphics[clip, width=5cm]{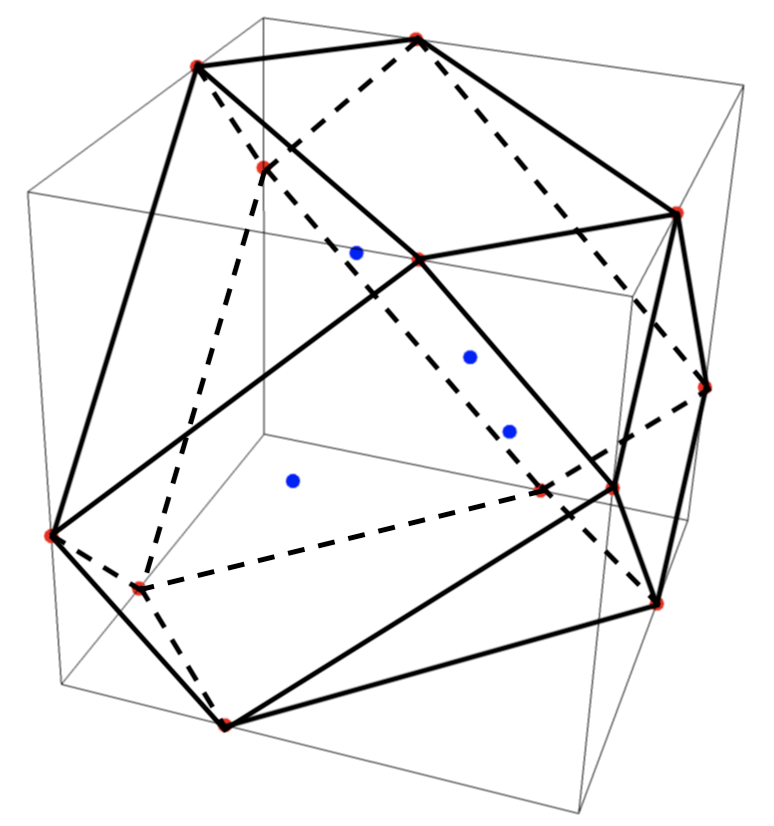}
         \caption{$co(V(1/2))$}
       \label{V1}
    
        \end{center}
      \end{minipage}

    \end{tabular}

  \end{center}
\end{figure}
\begin{lem}
\label{vert}
The convex hull of $V(c)$ is equal to the polyhedron whose vertices are $V^{\prime}(c)$. 
\end{lem}
\begin{proof}
We show that $co(V(c))\subset co(V^{\prime}(c))$.
We have that 
\begin{align*}
d_1:=\frac{1}{3}[(\frac{1}{1-c},\frac{-1}{1+c},\frac{1}{1-c})+(\frac{1}{1+c},\frac{-1}{1-c},\frac{1}{1-c})+(\frac{1}{1-c},\frac{-1}{1-c},\frac{1}{1+c})]\in co(V^{\prime}(c))
\end{align*}
and 
\begin{align*}
 d_2:=\frac{1}{3}[(\frac{-1}{1-c},\frac{1}{1-c},\frac{1}{1+c})+(\frac{-1}{1-c},\frac{-1}{1+c},\frac{-1}{1-c})+(\frac{1}{1+c},\frac{1}{1-c},\frac{-1}{1-c})]\in co(V^{\prime}(c)).
\end{align*}

Then 
\begin{align*}
a_{c,1}=\frac{1}{1+c}d_1+\frac{c}{1+c}d_2\in co(V^{\prime}(c)).
\end{align*}
Similarly, we have that $a_{c,2}, a_{c,3}, a_{c,4}\in co(V^{\prime}(c)).$ Hence we have that $co(V(c))\subset co(V^{\prime}(c))$.  

We fix each point $v\in V^{\prime}(c)$. Since the point $v$ belongs to an edge $e$ of $C_c$ and the other points in $V^{\prime}(c)$ belong to the edges which are different from $e$, we have that $v$ does not belong to the convex hull of $V^{\prime}(c)\backslash \{v\}$(See Figure \ref{V1}). Hence the point $v$ is a vertex of $co(V^{\prime}(c))=co(V(c)).$

Then we have proved our lemma.
\end{proof}
We now prove the following results.
\begin{thm}Let $c\in (0,1)$.
\begin{enumerate}
\renewcommand{\labelenumi}{(\arabic{enumi})}
\item The convex hull of $T(c)$ is a regular tetrahedron. The convex hull of $O(c)$ is a polyhedron like a cuboctahedron whose vertices are $V^{\prime}(c)$ (See the Figure \ref{a111}). 
\item $symA(c,P)= \mathcal{P}_4$.
\end{enumerate}
\end{thm}
\begin{proof}
\begin{enumerate}
\renewcommand{\labelenumi}{(\arabic{enumi})}
\item Fix $c\in (0,1)$.
As can be seen in Proposition 4.1, we have that $T(c)=A(c,E)$ and $O(c)=A(c,P)$, where $E$ is the identity matrix and $P$ is the $\pi$ rotation matrix around $l$. Here, $l$ is the unique line containing the origin and $(0,1,1)$. Note that $P^2=E$.

We consider $T(c)$. Note that $T(c)$ is the attractor of $\{f_i\}_{i\in I}$, where $f_i(x)=c(x-v_i/(1-c))+v_i/(1-c)$. By Lemma \ref{convex}, the convex hull of $T(c)$ is equal to the convex hull of $\{v_1/(1-c),v_2/(1-c),v_3/(1-c), v_4/(1-c)\}$, that is, a regular tetrahedron. 

We now consider $O(c)$. For each $(i,j)\in I^2$, we define the contracting mapping $g_{ij}:\mR \rightarrow \mR$ by 
\begin{align*}
g_{ij}(x)&=f_i\cdot{f_j}(x)\\&=cP(cPx+v_j)+v_i\\&=c^2x+cPv_j+v_i\\&=c^2(x-({v_i+cPv_j})/({1-c^2}))+({v_i+cPv_j})/({1-c^2}).
\end{align*}
Note that the set of points $\{({v_i+cPv_j})/({1-c^2})\}_{(i,j)\in I^2}$ is equal to $V(c)$.
Since the attractor of the IFS $\{g_{ij}\}_{(i,j)\in I^2}$ is equal to the attractor of the IFS $\{f_i\}_{i\in I}$, $O(c)$ is the attractor of $\{g_{ij}\}_{(i,j)\in I^2}$. By Lemma \ref{convex}, the convex hull of $O(c)$ is equal to the convex hull of $V(c)$. By Lemma \ref{vert}, we have that the convex hull of $O(c)$ is a polyhedron like a cuboctahedron whose vertices are $V^{\prime}(c)$.
Hence we have that the convex hull of $T(c)$ is a regular tetrahedron and the convex hull of $O(c)$ is a polyhedron like a cuboctahedron whose vertices are $V^{\prime}(c)$.

\item Fix $c\in(0,1), P\in{\mathcal{P}_6}$ and $Q\in{\mathcal{P}_4}$. We show that $\mathcal{P}_4\subset{symA}$. For each $x\in{A(c,P)}$, there exists a word $\omega=\omega_0\omega_1\omega_2\cdots\in{I^{\infty}}$ such that $x=\sum_{i=0}^{\infty}(cP)^{i}v_{\omega_{i}}$.
For each $i=1,2,3,...$, we set $v^{\prime}_i:=QP^{i}v_{\omega_i}$. If $P^{i}\in{\mathcal{P}_4}$, then $v^{\prime}_i\in{V}$. If $P^{i}\in{{\mathcal{P}_6}\backslash{\mathcal{P}_4}}$, then $v^{\prime}_i\in{V^{\prime}}$ since $P^{i}V=V^{\prime}$. Let $u=u_0u_1u_2\cdots\in{I^{\infty}}$ be the word such that $v_{u_0}=Qv_{\omega_0}\in{V},v_{u_{i}}:=P^{-i}v^{\prime}_i\in{V}$. Then we have 
\begin{align*}
Qx&=\sum_{i=0}^{\infty}c^iQP^{i}v_{\omega_{i}}\\
&=\sum_{i=0}^{\infty}c^iP^iv_{u_{i}}\in{A(c,P)}.
\end{align*}
Hence we have that $\mathcal{P}_4\subset{symA}$.

We consider the case $A(c,P)=T(c)$.
Suppose that there exists $Q^{\prime}\in SO(3)\backslash \mpt$ such that $Q^{\prime}T(c)=T(c)$. Then $Q^{\prime}co(T(c))=co(T(c))$. Since the vertices of $co(T(c))$ is $\{v_1/(1-c), v_2/(1-c), v_3/(1-c), v_4/(1-c)\}$, we have that $Q^{\prime}\{v_1/(1-c), v_2/(1-c), v_3/(1-c), v_4/(1-c)\}=\{v_1/(1-c), v_2/(1-c), v_3/(1-c), v_4/(1-c)\}$. But this contradicts $Q^{\prime}\in SO(3)\backslash \mpt$. Hence $symT(c)= \mathcal{P}_4$.

We next consider the case $A(c,P)=O(c)$.
Suppose that there exists $Q^{\prime}\in SO(3)\backslash \mpt$ such that $Q^{\prime}O(c)=O(c)$. Then $Q^{\prime}co(O(c))=co(O(c))$.
Since the vertices of $co(O(c))$ is $V^{\prime}(c)$, we have that $Q^{\prime}V^{\prime}(c)=V^{\prime}(c)$. 
We divide $C_c$ into the following eight cubes $C_1, C_2,...,C_8$. Here, 
\begin{align*}
&C_1:=\{(x, y, z)\in C_c|0\le x \le \frac{1}{1-c}, -\frac{1}{1-c}\le y \le 0, 0\le z \le \frac{1}{1-c}\}\\
&C_2:=\{(x, y, z)\in C_c|-\frac{1}{1-c}\le x \le 0, 0\le y \le \frac{1}{1-c}, 0\le z \le \frac{1}{1-c}\}\\
&C_3:=\{(x, y, z)\in C_c|-\frac{1}{1-c}\le x \le 0, -\frac{1}{1-c}\le y \le 0, -\frac{1}{1-c}\le z \le 0\}\\
&C_4:=\{(x, y, z)\in C_c|0\le x \le \frac{1}{1-c}, 0\le y \le \frac{1}{1-c}, -\frac{1}{1-c}\le z \le 0\}\\
&C_5:=\{(x, y, z)\in C_c|0\le x \le \frac{1}{1-c}, 0\le y \le \frac{1}{1-c}, 0\le z \le \frac{1}{1-c}\}\\
&C_6:=\{(x, y, z)\in C_c|-\frac{1}{1-c}\le x \le 0, -\frac{1}{1-c}\le y \le 0, 0\le z \le \frac{1}{1-c}\}\\
&C_7:=\{(x, y, z)\in C_c|-\frac{1}{1-c}\le x \le 0, 0\le y \le \frac{1}{1-c}, -\frac{1}{1-c}\le z \le 0\}\\
&C_8:=\{(x, y, z)\in C_c|0\le x \le \frac{1}{1-c}, -\frac{1}{1-c}\le y \le 0, -\frac{1}{1-c}\le z \le 0\}
\end{align*}
Note that $\bigcup_{i=1}^8C_i=C_c$. Since $V^{\prime}(c)$ is equal to 
\begin{align*}\{
&(\frac{1}{1-c},\frac{-1}{1+c},\frac{1}{1-c}),(\frac{1}{1+c},\frac{-1}{1-c},\frac{1}{1-c}),(\frac{1}{1-c},\frac{-1}{1-c},\frac{1}{1+c}),\\
&(\frac{-1}{1-c},\frac{1}{1+c},\frac{1}{1-c}),(\frac{-1}{1+c},\frac{1}{1-c},\frac{1}{1-c}),(\frac{-1}{1-c},\frac{1}{1-c},\frac{1}{1+c}),\\
&(\frac{-1}{1-c},\frac{-1}{1+c},\frac{-1}{1-c}),(\frac{-1}{1-c},\frac{-1}{1-c},\frac{-1}{1+c}),(\frac{-1}{1+c},\frac{-1}{1-c},\frac{-1}{1-c}),\\
&(\frac{1}{1-c},\frac{1}{1+c},\frac{-1}{1-c}),(\frac{1}{1+c},\frac{1}{1-c},\frac{-1}{1-c}),(\frac{1}{1-c},\frac{1}{1-c},\frac{-1}{1+c})\},
\end{align*} 
$V^{\prime}(c)\subset \bigcup_{i=1}^4 C_i\backslash {\bigcup_{i=5}^8 C_i}$. Since $Q^{\prime}\in SO(3)\backslash \mpt$, $Q^{\prime}\bigcup_{i=1}^4 C_i=\bigcup_{i=5}^8 C_i$. But this contradicts $Q^{\prime}V^{\prime}(c)=V^{\prime}(c)\subset \bigcup_{i=1}^4 C_i$.
\end{enumerate}
Hence we have proved our theorem.
\end{proof}
\begin{figure}[h]
  \begin{center}
    \begin{tabular}{c}

      \begin{minipage}{0.4\hsize}
        \begin{center}
          \includegraphics[clip, width=5cm]{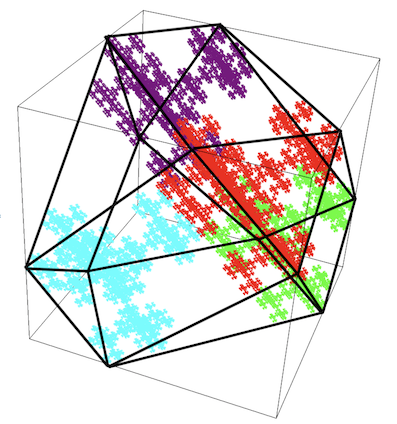}
         \caption{The convex hull of $O(1/2)$}
     \label{a111}

        \end{center}
      \end{minipage}
    \end{tabular}

  \end{center}
\end{figure}
It is natural to investigate the rotational symmetry of $A(c,P)$, where $P\in SO(3)\backslash \mpo$.  But it is difficult to investigate when $c$ is large. We can get the following Theorem.
\begin{thm}
Let $c\in(0,\frac{\sqrt{2}}{\sqrt{2}+\sqrt{3}}]$ and $P\in{SO(3)}\backslash{\mathcal{P}_6}$, then $symA(c,P)\neq{\mathcal{P}_4}$.
\end{thm}
\begin{proof}
We set $B=B(0,\frac{\sqrt{3}}{1-c})$ and $B_i:=f_i(B)=B(v_i,\frac{c\sqrt{3}}{1-c})$. Since $|v_i|+\frac{c\sqrt{3}}{1-c}=\sqrt{3}+\frac{c\sqrt{3}}{1-c}=\frac{\sqrt{3}}{1-c}$, $\overline{B_i}\subset{\overline{B}}$. Since $|v_i-v_j|=2\sqrt{2}\ge \frac{2\cdot c\sqrt{3}}{1-c}$, we have for all $i\neq{j}$, $B_i\cap{B_j}={\emptyset}$. Note that $A_i\subset{\overline{B_i}}$, where $A_i:=f_i(A)$.

For each $Q\in{\mathcal{P}_4}$, $Q$ permutes $\{B_i\}_{i=1}^4$. Hence $Q$ induces a permutation of the indices $\{1,2,3,4\}$. Let $\ast: \mpt\rightarrow S_4$ be the mapping defined by $Q(B_i)=B_{Q^{\ast}i}$ for each $i$. Note that $\ast: \mpt\rightarrow S_4$ is a group homomorphism, where $S_4$ denotes the symmetric group of all permutations of $\{1,2,3,4\}$.

Suppose $symA=\mathcal{P}_4$. For each $Q\in{\mathcal{P}_4}$,
\begin{align*}
QA=A
&\Rightarrow{Q(\bigcup_{i=1}^4{A_i})=\bigcup_{i=1}^4A_i}\\
&\Rightarrow{Q(\bigcup_{i=1}^4{A_i})\cap{QB_i}=\bigcup_{i=1}^4{A_i}\cap{QB_i}}(\mbox{for\ all\ }i=1,2,3,4)\\
&\Rightarrow{Q(A_i\cap{B_i})=A_{Q^{\ast}i}\cap{B_{Q^{\ast}i}}}(\mbox{for\ all\ }i=1,2,3,4)\\
&\Rightarrow{QA_i=A_{Q^{\ast}i}}(\mbox{for\ all\ }i=1,2,3,4)\\
&\Rightarrow{{f_{Q^{\ast}i}}^{-1}Qf_i\in{{\rm sym}(A)}=\mathcal{P}_4}.
\end{align*}

Where, we use $\overline{A\cap{B}}=A$ to imply the forth $\Rightarrow$. 

We prove $\overline{A\cap{B}}=A$. 

First, we prove $A\cap B\neq{\emptyset}$. Suppose that $A\cap B={\emptyset}$. Since $A\subset \overline{B}$, we have that $A\subset \partial B$. Since $\#A\ge2$, there exits $a_1\neq a_2\in A(\subset \partial B)$. Then we have that $f_1(a_1),f_1(a_2)\in A_1\subset{\overline{B_1}}$. Since $\overline{B_1}\subset{\overline{B}}$ , $\#\overline{B_1}\cap \partial{B}=1$ and $f_1(a_1),f_1(a_2)\in \partial B$, we have that$f_1(a_1)=f_1(a_2)$. But this contradicts that $f_1$ is injective. Hence we have that $A\cap B\neq{\emptyset}$.

Hence there exists $v\in{A\cap B}$. Since for each $\omega=\omega_1\omega_2\cdot\cdot\cdot\omega_n$, $f_{\omega}(A)\subset{A}$ and $f_{\omega}(B)\subset{B}$, we have that $f_{\omega}(v)\in{A\cap B}$. Hence $\bigcup_{n=1}^{\infty}\bigcup_{\omega\in I^n}f_{\omega}(v)\subset{A\cap B}$. Furthermore, it is well known that $A\subset\overline{\bigcup_{n=1}^{\infty}\bigcup_{\omega\in I^n}\{f_{\omega}(v)\}}$. Hence we have that $A\subset \overline{A\cap{B}}$. Thus we have that $\overline{A\cap{B}}=A$.

Since ${f_{Q^{\ast}i}}^{-1}Qf_i(x)=P^{-1}QP(x)$ and ${f_{Q^{\ast}i}}^{-1}Qf_i\in{{\rm sym}(A)}=\mathcal{P}_4$, there exists $Q^{\prime}\in{\mathcal{P}_4}$ such that $QP=PQ^{\prime}$.

For each $i=1,2,3,4$, let $l_i$ be the line containing the origin and $v_i$. Since $P\notin{\mathcal{P}_6}$, we have that $\{l_1,l_2,l_3,l_4\}\neq \{Pl_1, Pl_2, Pl_3, Pl_4\}$, and hence $symT_0\neq symPT_0$. whereas since $Q(PT_0)=QP(T_0)=PQ^{\prime}(T_0)=PT_0$, we have that $Q\in symPT_0$. Hence $symT_0\subset symPT_0$. Since $\#symPT_0=\#symT_0=12$, we have that $symT_0=symPT_0$, but this contradicts $symT_0\neq \rm symPT_0$. Hence we have proved our theorem.
\end{proof}
Let $L$ be a unique line containing the origin and $(0,0,1)$. Next, we consider the connectedness of $A(c,P)$, where $c\in(0,1)$ and $P$ is the rotation matrix around $L$.  
We set $\Delta=\{v_i-v_j|i,j\in\{1,2,3,4\}\}$. 
\begin{lem}
\label{conn}
Let $c\in(0,1)$ and let $P$ be a rotation matrix around $L$. Then $A(c,P)$ is connected if and only if ${A(c,P)_1\cap{A(c,P)_3}\neq{\emptyset}}$.
\end{lem}
\begin{proof}
We set $A=A(c,P)$. First, we show 
\begin{align}
A_1\cap{A_3}\neq{\emptyset}\Leftrightarrow{A_1\cap{A_4}\neq{\emptyset}}\Leftrightarrow{A_2\cap{A_4}\neq{\emptyset}}\Leftrightarrow{A_2\cap{A_3}\neq{\emptyset}}
\end{align}
Suppose $A_1\cap{A_3}\neq{\emptyset}$. Then there exists $\{d_i\}_{i=1}^{\infty}\in{\Delta}^{\infty}$ such that 
\begin{align*}
v_1-v_3+\displaystyle\sum_{i=1}^{\infty}{(cP)^{i}d_i}=0.
\end{align*}
 Let $Q$ be the $-\pi/2$ rotation matrix around $L$. We have  
\begin{align*}
0
&=Q(v_1-v_3+\displaystyle\sum_{i=1}^{\infty}{(cP)^{i}d_i})\\
&=Q(v_1-v_3)+Q(\displaystyle\sum_{i=1}^{\infty}{(cP)^{i}d_i})\\
&=v_1-v_4+\displaystyle\sum_{i=1}^{\infty}{Q(cP)^{i}d_i}\\
&=v_1-v_4+\displaystyle\sum_{i=1}^{\infty}{(cP)^{i}Qd_i}.
\end{align*}

Furthermore, ${Q\Delta}\subset{\Delta}$. Hence we have that $A_1\cap{A_4}\neq{\emptyset}$.

By using the same method above, we can show that $A_1\cap{A_4}\neq{\emptyset}$ implies $A_2\cap{A_4}\neq{\emptyset}$, $A_2\cap{A_4}\neq{\emptyset}$ implies $A_2\cap{A_3}\neq{\emptyset}$ and $A_2\cap{A_3}\neq{\emptyset}$ implies $A_1\cap{A_3}\neq{\emptyset}$.  Hence we have proved (1). 

Suppose $A(c,P)_1\cap A(c,P)_3\neq \emptyset$. We show that $A$ is connected. By\cite{Hata}, \cite{kigami}, it suffices to show that for all $(i,j)\in \{1,2,3,4\}^2$,
\begin{align}
\mbox{there\ exist\ } n_1,...,n_k\in \{1,2,3,4\}\ \mbox{such that}\ n_1=i, n_k=j\ \mbox{and}\ {A_{n_l}}\cap{A_{n_{l+1}}}\neq{\emptyset}\ (l=1,...,k).
\end{align}
By (4), we can show that for each $(i,j)\in \{1,2,3,4\}^2$, (5) holds. Hence $A$ is connected.

Finally, suppose that $A$ is connected and $A_1\cap{A_3}={\emptyset}$. By $\cite{Hata}, \cite{kigami}$, (5) holds in the case of $(i,j)=(1,3)$. But this contradicts (4). Hence we have proved our lemma.
\end{proof}
\begin{cor}
Let $c\in [1/2,1)$ and $P\in \mpo$. Then $A(c,P)$ is connected.
\end{cor}
\begin{proof}
Let $c\in [1/2,1)$ and $E$ be the identity matrix. Then $A(c,E)=T(c)$. Since the attractor of the IFS on $\{cx+v_1,cx+v_3\}$ on $\mR$ is a line segment, there exists $\{d_i\}_{i=1}^{\infty}(d_i=v_1-v_3\ or\ v_3-v_1\ or\ 0)$ such that 
\begin{align*}
v_1-v_3+\displaystyle\sum_{i=1}^{\infty}{c^{i}d_i}=0.
\end{align*}
By Lemma \ref{conn}, $T(c)$ is connected. Let $P$ be the $\pi/2$ rotation matrix around $L$. Note that $P\in {{\mathcal{P}_6}\backslash{\mathcal{P}_4}}$. For all $i=1,2,3,...$, we set $d^{\prime}_i=P^{-1}d_i\in \Delta$. Then we have 
\begin{align*}v_1-v_3+\displaystyle\sum_{i=1}^{\infty}{(cP)^{i}d^{\prime}_i}&=v_1-v_3+\displaystyle\sum_{i=1}^{\infty}{c^{i}d_i}\\
&=0.
\end{align*}
Hence $O(c)$ is connected.
Hence we have proved our corollary. 
\end{proof}
　

\noindent Yuto Nakajima\\
E-mail: nakajima.yuuto.32n@st.kyoto-u.ac.jp

\end{document}